\documentclass{article}
\usepackage{fullpage}
\usepackage[utf8]{inputenc}
\usepackage{amssymb,amsmath,amsthm,mathtools,enumitem}
\usepackage{ytableau,tikz}
\usepackage[msc-links]{amsrefs}
\usetikzlibrary{decorations.pathreplacing}

\usepackage{forest}
\usepackage{hyperref}
\usepackage{cases}
\usepackage{todonotes}

\usepackage{tikzsymbols}

\theoremstyle{plain}
\newtheorem{theorem}{Theorem}[section]

\newtheorem{corollary}[theorem]{Corollary}
\newtheorem{lemma}[theorem]{Lemma}
\newtheorem{remark}[theorem]{Remark}

\theoremstyle{definition}
\newtheorem{definition}[theorem]{Definition}

\newtheorem{definitions}[theorem]{Definitions}
\newtheorem{conjecture}[theorem]{Conjecture}
\newtheorem{example}[theorem]{Example}
\newtheorem*{theorem*}{Theorem}

\DeclarePairedDelimiterX\makeset[2]{\{}{\}}{#1\; \delimsize\vert\; #2}

\newcommand{\Av}{\text{Av}}
\newcommand{\N}{\mathcal{N}}

\newcommand{\E}{\mathcal{E}}
\renewcommand{\P}{\mathcal{P}}
\newcommand{\Q}{\mathcal{Q}}
\renewcommand{\S}{\mathcal{S}}

\renewcommand{\ne}[1]{\overline{#1}}

\DeclareMathSymbol{\shortminus}{\mathbin}{AMSa}{"39}

\begin{document}

\title{Counting pattern-avoiding integer partitions}

\author{
  Jonathan Bloom\\
  \texttt{Lafayette College}\\
  \texttt{bloomjs@lafayette.edu}\\
  \and 
  Nathan McNew\\
    \texttt{Towson University}\\
    \texttt{nmcnew@towson.edu}
}
\date{\vspace{-3ex}}

\maketitle
\begin{abstract}
A partition $\alpha$ is said to contain another partition (or pattern) $\mu$ if the Ferrers board for $\mu$ is attainable from $\alpha$ under removal of rows and columns. We say $\alpha$ avoids $\mu$ if it does not contain $\mu$.  In this paper we count the number of partitions of $n$ avoiding a fixed pattern $\mu$, in terms of generating functions and their asymptotic growth rates.  

We find that the generating function for this count is rational whenever $\mu$ is (rook equivalent to) a partition in which any two part sizes differ by at least two.  In doing so, we find a surprising connection to metacyclic $p$-groups.  We further obtain asymptotics for the number of partitions of $n$ avoiding a pattern $\mu$.  Using these asymptotics we conclude that the generating function for $\mu$ is not algebraic whenever $\mu$ is rook equivalent to a partition with distinct parts whose first two parts are positive and differ by 1.  
\end{abstract}

\maketitle

\section{Introduction}\label{sec:intro}
In \cite{Bloom:Rook-2018} the first author and Saracino introduced the following notion of pattern-avoiding integer partitions.  
Viewing two integer partitions $\alpha$ and $\mu$ as Ferrers boards, we say that $\alpha$ \emph{contains} $\mu$ if there exist some set of rows and columns that can be deleted from $\alpha$ so that, after top/left justifying the remaining boxes, we obtain $\mu$. If this is not possible we say that $\alpha$ \emph{avoids} $\mu$.  We denote by $\Av_n(\mu)$ the set of all $\mu$-avoiding partitions of $n\geq 0$ and set $\Av(\mu) = \cup_{n\geq 0}\Av_n(\mu)$.  

For example, $\alpha = (6,5,5,5,4,4,2,2)$ contains $\mu=(4,3,3,2,2)$ since we can delete the rows and columns indicated in red below and then justify the remaining boxes to obtain $\mu$.  

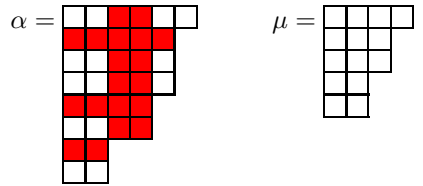
\ytableausetup{boxsize=.8em}
$$\alpha=\ydiagram{6,5,5,5,4,4,2,2}*[*(red)]{0,5,0,0,4,0,2}*[*(red)]{2+2,2+2,2+2,2+2,2+2,2+2}\hspace{1cm} \mu = \ydiagram{4,3,3,2,2}  \ .$$
Additionally, when $\mu = (2,1)$ then $\Av(\mu)$ consists of all partitions whose Ferrers board is a rectangle.  

The purpose of this paper is to study the sequence 
\begin{equation}\label{seq}
|\Av_1(\mu)|, |\Av_2(\mu)|,|\Av_3(\mu)|,\ldots\ 
\end{equation}
of the counts of partitions avoiding a fixed partition $\mu$. To do so, we investigate the generating function
\begin{equation}\label{gf}
    \sum_{n\geq 0}|\Av_n(\mu)|z^n,
\end{equation}
which we casually refer to as \emph{the generating function for $\mu$}, as well as the asymptotic growth rate of \eqref{seq}.

Throughout the paper we regard a partition $\mu$ of $n\geq 1$ as an infinite weakly decreasing sequence of nonnegative integers $\mu_1\geq \mu_2\geq \cdots$, called \emph{parts}, whose nonzero terms sum to $n$.    We call $n$ the \emph{weight} of $\mu$ and write $|\mu|=n$.  The set of all partitions is denoted by $\P$. We call a partition \emph{strict} provided that all its positive parts are distinct and \emph{super-strict} if all its positive parts differ by at least 2.

While far from obvious, previous work shows we can restrict our attention to avoiding strict partitions $\mu$ without any loss of generality. In particular the first author and Saracino show \cites{Bloom:Rook-2018,Bloom:On-cr2018} the following.

\begin{theorem}\label{thm:strict}
For any partition $\tau$ there exists a unique strict partition $\mu$ such that $|\Av_n(\tau)| = |\Av_n(\mu)|$ for all $n\geq 0$.
\end{theorem}
We say two partitions $\mu$ and $\tau$ are \emph{Wilf equivalent} if $|\Av_n(\mu)| = |\Av_n(\tau)|$ for all $n\geq 1$. (Our choice of the term ``Wilf equivalence'' is in reference to a similar definition found in the theory on pattern-avoiding permutations.)  More background on this is described in Section \ref{sec:background}.

In light of Theorem~\ref{thm:strict}, we consider only the problem of avoiding strict partitions $\mu$ for the remainder of the paper. We start, in Section~\ref{sec:rational}, by considering the generating function \eqref{gf} when $\mu$ is strict.  We show, in Theorem \ref{thm:rational} that this generating function is rational whenever $\mu$ is super-strict.

As a corollary (Corollary \ref{cor:K}) we conclude that for any fixed $K\geq 1$ the generating function counting all partitions  $\alpha$ with the property that any two (positive) parts of $\alpha$ differ by at most $K$ is rational.  The proof of Theorem \ref{thm:rational} is constructive in that it gives a method to compute the rational function associated to the generating function, which we describe further at the end of the section.  

Implementing this method we find the generating functions for various small partitions $\mu$, the results of which, along with corresponding asymptotics, are tabulated in Table \ref{tab:my_label}.  In creating this table, we discovered, by way of the OEIS database, that the generating function for $\mu=(5,2)$ is identical to the generating function for the number of so-called metacyclic $p$-groups for prime $p$.  We record this curious coincidence and give more details in Remark~\ref{rmk:metacyclic}.  

In Section~\ref{sec:growth} we find the asymptotic growth rate of the sequence (\ref{seq}) for every strict partition $\mu$. We find that this growth rate has a different form for staircase partitions of the form $\mu=(k+1,k,\ldots, 1)$ than it does for other sorts of strict partitions. Using results of Ingham \cite{Ingham} and Estermann 
\cites{Estermann1,Estermann2} we prove in Theorem~\ref{thm:staircase} that

\begin{equation*}
|\Av_n(k+1,k,\ldots,1)| =\begin{cases}
\sigma_0(n) & k=1  \\
\frac{1}{2\zeta(2)}\sigma_{1}(n)\log^2 n\left(1+O\left(\frac{\log \log n}{\log n}\right)\right) & k=2 \\
\frac{1}{k!(k-1)!\zeta(k)}\sigma_{k-1}(n)\log^k n\left(1+O\left(\frac{1}{\log n}\right)\right) & k\geq 3,
\end{cases}
\end{equation*}
where $\sigma_i(n) = \sum_{d\vert n}d^i$.  If $\mu$ is not a staircase, then we can write 
$$\mu = (k+1,k, k-1,\ldots, k-\ell+1, a_0,a_1,\ldots). $$
where $k-\ell>a_0$, so that $k - \ell$ is the first size omitted from $\mu$.  In Theorem~\ref{thm:not staircase} we show that for such a partition $\mu$ we have
\begin{equation} |\Av_n(\mu)| 
= \frac{n^{k-1}\log^\ell n}{\ell!(k-1)!
\prod_{j=0}^{k-\ell-1}\left((k-\ell) -(a_j+j)\right)}
\left(1+O\left(\frac{1}{\log n}\right)\right). \label{eq:introgencase} \end{equation}

To facilitate our study of partitions that are not staircases we consider first, in Section \ref{sssec:nonstaircase 1}, the special case of partitions whose first two parts differ by at least 2.  In this case we show in Theorem~\ref{thm:nonotch} that \eqref{eq:introgencase} holds with a much stronger error term.

From Theorems~\ref{thm:staircase} and \ref{thm:not staircase} we see that the leading term of our asymptotic expression contains a log factor whenever the largest two parts of $\mu$ are positive and differ by 1.  This suggests that such generating functions cannot be rational.  In fact, we prove in  Corollary~\ref{cor:not rational} the stronger result that such generating functions cannot be algebraic.  Based on this, as well as computations of the lower order terms in the asymptotics for various small strict partitions we conjecture that in fact more is true.

\begin{conjecture}
If $\mu$ is any strict partition that is not super-strict, i.e., $\mu$ contains two nonzero parts which differ by 1, then the generating function for $\mu$ is not algebraic.
\end{conjecture}

\section{Background} \label{sec:background}
It is worth recalling several results from the literature which put our results in context.  We begin with one from the theory of rooks.  

The theory of rooks began in the late 1940's with the paper \cite{kaplansky1946problem} by Kaplansky and Riordon.  In this paper the authors introduced the idea of a rook polynomial as a tool for understanding permutations which avoid certain positions, e.g., derangements.  Simply, a rook polynomial for a partition $\mu$ is the polynomial whose coefficient on $z^k$ is the number of ways to place $k\geq 0$ ``rooks'' on the Ferrers board for $\mu$ so that no two rooks are in the same row and column.  Some twenty years later Foata and Sch\"{u}tzenberger~\cite{foata1970rook} made the following definition.  Two partitions are said to be rook equivalent provided that their rook polynomials are equal. In this same paper Foata and Sch\"{u}tzenberger proved that each rook-equivalence class contains a unique strict partition.  

At this point the theory of rooks merges with our study of pattern-avoiding integer partitions.  The first author and Saracino in \cite{Bloom:Rook-2018} prove that if two partitions are rook equivalent then they are also Wilf equivalent.  The same two authors, shortly thereafter, prove \cite{Bloom:On-cr2018} the reverse implication thereby establishing that rook equivalence coincides with Wilf equivalence.  Combining this with the result of Foata and Schu\"{u}tzenberger mentioned above implies that each Wilf-equivalence class contains a unique strict partition. 

Another connection of our work to the theory of rooks can be found lurking in the constant for the leading term when $\mu$ is not a staircase.  In particular, the factors in the product  $$\prod_{j=0}^{k-\ell-1}\left((k-\ell) -(a_j+j)\right)$$ 
from \eqref{eq:introgencase} are reminiscent of another theorem due to Foata and Sch\"utzenberger.  In \cite{foata1970rook} these authors provide a beautiful characterization of rook equivalence by proving that partitions $\mu$ and $\nu$ are rook equivalent if and only if we have the following equality of multisets 
$$\{1+\mu_1, 2+\mu_2,\ldots\}=\{1+\nu_1, 2+\nu_2,\ldots\}.$$  
Rewriting the above product as 
$$\prod_{j=0}^{k-\ell-1}\left(k-\ell+1 -(a_j+j+1)\right)= \prod_{i=1}^{k-\ell}\left(k-\ell+1 -(\nu_i+i)\right),$$
where $\nu = (a_0,a_1,\ldots)$, we see that our product involves numbers of the form used by Foata and Sch\"utzenberger in their classification theorem.  At this time we are unaware of the precise significance of this observation.  That said, it comes as little surprise that these numbers appear in our results since they appear in much of the rook theory literature.  For example, these numbers are heavily used in the papers   \cite{Bloom:Rook-2018} and \cite{Bloom:On-cr2018} to prove that rook equivalence is the same as Wilf equivalence.  Such numbers also appear in the beautiful result from \cite{goldman1975rook} of Goldman, Joichi, and White where they show that rook polynomials in the falling factorial basis factor entirely with roots that are essentially these numbers.

Partitions avoiding certain specific patterns have been previously studied.  MacMahon \cite{MacMahon} considered partitions with $k$ distinct magnitudes, which are exactly the set  $\Av\big((k+1,k,k-1,\ldots,1)\big) \setminus \Av\big((k,k-1,\ldots,1)\big)$ of partitions avoiding a staircase of size $k+1$ but containing a staircase of size $k$.  In particular, he found generating functions for the number of partitions in this set.  

These partitions were further considered by Andrews \cite{AndrewsSLB} who shows that the number of such partitions with weight $n$ is asymptotic to ($\frac{1}{k!}$ times) the $k$-fold convolution sum of divisor functions.  This convolution sum also has a rich history.  The convolution $\sum_{i=1}^n \sigma_0(i)\sigma_0(n-i)$ was considered by Ingham, who found an asymptotic expression for this sum by elementary means.  Subsequently, Estermann, in a series of papers, found lower order terms for both this and the $k$-fold convolution of divisor functions using the circle method.  

Finally, the set of partitions into at most $k$ parts, which is (up to conjugation) the set of $(k+1)$-avoiding partitions has been extremely well studied in the literature going back to at least Euler. We make no attempt to recount that history here.  Instead we point out that in terms of our definition of pattern avoidance the previous literature involves only the two extreme cases, i.e., partitions avoiding the largest and smallest partitions with largest part $k+1$.  The results in this paper can therefore be viewed as an interpolation between these two extremes.

\section{Generating functions of super-strict partitions}\label{sec:rational}

We start with a few straightforward definitions. 
\begin{definition}
For partitions $\tau$ and $\mu$, we define $\Q_n(\tau,\mu)$ to be the set of all partitions with weight $n$ that contain $\tau$ and avoid $\mu$. We set $\Q(\tau,\mu) = \bigcup_{n\geq 1} \Q_n(\tau,\mu)$.
\end{definition}

\begin{definition}
For a partition $\alpha$ let $m(\alpha)$ to be the multiplicity of $\alpha_1$ in $\alpha$, i.e., the length of the rightmost column of $\alpha$ when viewed as a Ferrers board. For any nonempty set  $\S$ of partitions define 
$$F_{\S}(z,t) = \sum_{\alpha\in \S} z^{|\alpha|}t^{m(\alpha)}.$$
If $\S=\emptyset$ set $F_\S = 0$.  In the case when $S = \Av(\mu)$ or $S = \Q(\tau,\mu)$ we abuse notation and instead write $F_\mu(z,t)$ or $F_{\tau,\mu}(z,t)$, respectively.  	
\end{definition}

We shall also need the following familiar notion.  The \emph{southeast border} of a partition $\mu$ is the lattice path consisting of ``north'' and ``east'' steps which traces along the bottom/right of the columns/rows in $\mu$.  For example, the southeast border for 
$$\mu = \ydiagram{7,4,3}$$ 
is $(e,e,e,n,e,n,e,e,e,n)$ where $e$ and $n$ represent east and north steps, respectively. Certainly every southeast border starts with an east step and ends with a north step. We call a north step followed immediately by an east step a \emph{north-east step}.  Strict partitions are precisely those whose southeast border does not contain consecutive north steps.  Consequently, the southeast border of a strict partition can be written as a sequence of east and north-east steps with a final north step.   Doing this for the above example we get the sequence $(e,e,e,n',n',e,e,n)$ where $n'$ denotes a north-east step.

Armed with these basic definitions the goal of this section is to show that $F_\mu(z,1)$ is rational when $\mu$ is super-strict and establish an algorithm for computing  $F_\mu(z,1)$ in this case. To this end we first define two operators $\E$ and $\N$ which shall correspond to the east and north-east steps, respectively, in the southeast border of $\mu$.  We then show that by mapping the southeast border of $\mu$ to a composition of such operators, we obtain a function $\Theta_\mu$ with the property that 
$$F_\mu(z,1) = \left.\Theta_\mu\left(\frac{zt}{1-zt}\right)\right|_{t=1}.$$

With this in mind let us start by defining the operators $\E$ and $\N$.  We hold off on the motivation for these definitions as it is not required at the moment.  Instead the value of these definitions shall become apparent below with the statement and proof of  Lemma~\ref{lem:GF and E/N}. 

\begin{definition}
Let $\displaystyle G(z,t) = \sum_{n\geq 1}\sum_{m\geq 0} a_{n,m} z^nt^m$.  Then define 
$$\E G(z,t) = \frac{G(z,1)-ztG(z,zt)}{1-zt}$$
and 
$$\N G(z,t) = G(z,0)+\sum_{n\geq 1}\sum_{m\geq 1} a_{n,m} \left(\frac{1}{1-z^m}\right)\left(\frac{1-(tz)^m}{1-tz}\right)z^n.$$
\end{definition} 

Note that if $G(z)$ is a formal power series not dependent on $t$, then $\E G(z) = G(z)$ and $\N G(z) = G(z)$.

\begin{definition}
For partitions $\alpha$ and $\beta$ we set
$$\alpha+\beta:= (\alpha_1+\beta_1, \alpha_2+\beta_2,\ldots).$$ 
\end{definition}
\ytableausetup{boxsize=.8em}	\noindent In terms of Ferrers boards, this sum is the partition whose columns are those of $\alpha$ together with those of $\beta$.

\begin{definitions}
For any super-strict partition $\mu$ define the sequence $\Theta_\mu$ in the symbols $\E$ and $\N$ as follows.  Reading the southeast border of $\mu$ from bottom/left to top/right and ignoring the initial east step, the final two east and north steps, record each east step by an $\E$ and each north-east step by an $\N$.  

Additionally, for any nonempty sequence $\Theta=(\Theta_1,\Theta_2,\ldots)$ in the symbols $\E$ and $\N$ we define
$$\Theta G(z,t): = \cdots \Theta_2\circ \Theta_1 G(z,t).$$
If $\Theta = \varnothing$ is the empty sequence, then we define $\Theta G(z,t) = G(z,t)$.
\end{definitions}

\begin{example}
Taking
$$\mu = \ydiagram{8,5,3}$$ 
then $\Theta_\mu = (\E, \E, \N,\E,\N,\E)$.
\end{example}

We are now able to write the statement of the first theorem in this section.  

\begin{theorem}\label{thm:F}
Let $\mu$ be a super-strict partition with weight at least 2.  Let $\tau$ be such that $\mu=\tau+(1)$.  Then,
$$\Theta_\mu\left(\frac{zt}{1-zt}\right) =F_{\tau}(z,1) + F_{\tau,\mu}(z,t).$$
In particular $F_\mu(z,1) = \left.\Theta_\mu\left(\frac{zt}{1-zt}\right)\right|_{t=1}$.  
\end{theorem}

Let us pause to illustrate this theorem.  In the case that $\mu=(2)$ we have $\Theta_{(2)} = \emptyset$, since the initial east step as well as the final two east and north steps are ignored, so 
$$\left.\left(\frac{zt}{1-zt}\right)\right|_{t=1} = \frac{z}{1-z}$$
which agrees with $F_{(2)}(z,1)$ in this case since $\Av\big((2)\big)$ consists of only single columns.   Next consider the case when $\mu = (3)$.  Here $\Theta_{(3)} = (\E)$ and  after some simplification we have
$$\E \left.\left(\frac{zt}{1-zt}\right)\right|_{t=1} = \frac{z+z^2 - z^3}{(1-z)(1-z^2)}= \frac{1}{(1-z^2)(1-z)}-1.$$
On the other hand, $\Av\big((3)\big)$ consists of all partitions with at most two columns which is easily seen to be counted by the expression on the right side.  

Returning to our main argument we aim to prove this theorem by induction on $|\mu|$.  As such  we shall require an understanding of how $\E$ and $\N$ each affect certain generating functions.  To this end we make the following definitions.

\begin{definitions}\label{def:EMN}
For any partition $\alpha$ we define
\begin{align*}
    E(\alpha) &= \makeset{\alpha +(1^c)}{0< c\leq m(\alpha)},\\
    M(\alpha) &= \makeset{\alpha + (w^{m(\alpha)} )}{0\leq w},\\
    N(\alpha) &= \makeset{\alpha + (w^{m(\alpha)} )+(1^c) }{0\leq w\text{ and } 0<c< m(\alpha)}.
\end{align*}
\end{definitions}
The reader should take note that in the definition of $E$ the upper bound on $c$ is $\leq$ while the upper bound on $c$ in the definition of $N$ is $<$.  Therefore the action of $N$ must create a nonempty rightmost column that is strictly shorter than the rightmost column in $\alpha$.  To illustrate consider the following examples.  

\begin{example}
Let $\alpha = (3,3,3,2)$.  Then $E(\alpha)$ consists of the partitions:
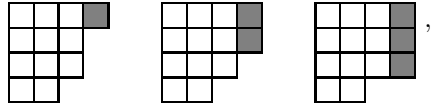
\ytableausetup{boxsize=.9em}
$$\ydiagram{3,3,3,2}*[*(gray)]{3+1}\qquad\ydiagram{3,3,3,2}*[*(gray)]{3+1,3+1}\qquad\ydiagram{3,3,3,2}*[*(gray)]{3+1,3+1,3+1}\ ,$$
$M(\alpha)$ is the set consisting of all partitions form:
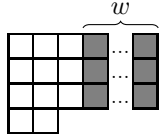
\ytableausetup
 {boxsize=.9em}
\begin{center}
\begin{tikzpicture}
\node  at (0,0) {
\begin{ytableau}
{} & {} & {} & *(gray){}&\none[\scriptstyle\dots]& *(gray){}\\
{} & {} & {} & *(gray){}&\none[\scriptstyle\dots]& *(gray){}\\
{} & {} & {} & *(gray){}&\none[\scriptstyle\dots]& *(gray){}\\
{} & {} \\
\end{ytableau}
};
\draw[decoration={brace,raise=4pt},decorate] (0,.6) -- (1,.6) node [above= 0.2, pos = 0.5] {$w$};
\end{tikzpicture} 
\end{center}
and $N(\alpha)$ consists of all partitions of the form:
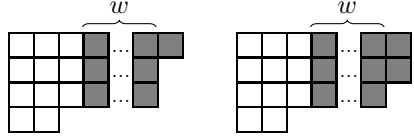
\begin{center}
\begin{tikzpicture}
\node  at (0,0) {
\begin{ytableau}
{} & {} & {} & *(gray){}&\none[\scriptstyle\dots]& *(gray){} & *(gray){}\\
{} & {} & {} & *(gray){}&\none[\scriptstyle\dots]& *(gray){}\\
{} & {} & {} & *(gray){}&\none[\scriptstyle\dots]& *(gray){}\\
{} & {} \\
\end{ytableau}};
\draw[decoration={brace,raise=4pt},decorate] (-.18,.6) -- (0.8,0.6) node [above= 0.2, pos = 0.5] {$w$};
\end{tikzpicture} 
\quad
\begin{tikzpicture}
\node  at (0,0) {
\begin{ytableau}
{} & {} & {} & *(gray){}&\none[\scriptstyle\dots]& *(gray){} & *(gray){}\\
{} & {} & {} & *(gray){}&\none[\scriptstyle\dots]& *(gray){} & *(gray){}\\
{} & {} & {} & *(gray){}&\none[\scriptstyle\dots]& *(gray){}\\
{} & {} \\
\end{ytableau}};
\draw[decoration={brace,raise=4pt},decorate] (-.18,.6) -- (0.8,0.6) node [above= 0.2, pos = 0.5] {$w$};
\end{tikzpicture} 
\end{center}
\end{example}

The relevance of these three functions is established by the following two lemmas.  For readability, we relegate their proofs to Subsection~\ref{sec:lemmas 1 and 2}.

\begin{lemma}\label{lem:E}
	Let $\mu$ be a partition, then 
\begin{equation}\label{eq:E}
E\Q(\mu, \mu+(1)) = \Q(\mu+(1),\mu+(2)).	
\end{equation}
	Further, for any $\beta$ in $\Q(\mu+(1), \mu+(2))$, there exists a unique $\alpha \in \Q(\mu, \mu+(1))$ such that $\beta \in E(\alpha)$.  
\end{lemma}

The following definition will be useful in stating the next lemma as well as throughout the remainder of this section.

\begin{definition}
	For any partition $\mu$ we define $\ne{\mu} = (\mu_1+1,\mu_1,\mu_2,\ldots)$.  
\end{definition}

\begin{lemma}\label{lem:N}
	Let $\mu$ be a strict partition.  Then 
\begin{equation}\label{eq:M}
M\Q(\mu,\mu+(1)) =  \Q(\mu,\ne\mu)
\end{equation}
and 
\begin{equation}\label{eq:N}
N\Q(\mu,\mu+(1)) =  \Q(\ne\mu,\ne\mu+(1)).
\end{equation}
Further, for any $\beta\in\Q(\mu,\ne\mu)$, respectively $\beta \in \Q(\ne\mu,\ne\mu+(1))$, there exists a unique $\alpha \in \Q(\mu,\mu+(1))$ such that $\beta \in M(\alpha)$, respectively $\beta \in N(\alpha)$.    
\end{lemma}

Our next lemma gives the promised explanation as to how $\E$ and $\N$ affect our generating functions. 

\begin{lemma}\label{lem:GF and E/N}
	Let $\mu$ be strict and set $\S = \Q(\mu, \mu+(1))$.  Then
	\begin{equation}\label{eq:scriptE/E}
	  \E F_{\S}(z,t) = F_{\S}(z,1) + F_{E\S}(z,t)  
	\end{equation}
and 
\begin{equation}\label{eq:scriptN/N}
    \N F_{\S}(z,t) = F_{M\S}(z,1) + F_{N\S}(z,t).
\end{equation}
\end{lemma}

\begin{proof}
	We first establish (\ref{eq:scriptE/E}).  By the uniqueness clause in Lemma~\ref{lem:E} we have 
	$$E\S = \bigsqcup_{\alpha\in \S} E(\alpha),$$
	where $\sqcup$ denotes disjoint union.  Consequently (\ref{eq:scriptE/E}) is equivalent to 
$$\E F_{\S}(z,t) = F_{\S}(z,1) + \sum_{\alpha\in \S} F_{E(\alpha)}(z,t).$$
	As such it suffices to consider the contribution of a single $\alpha\in \S$ to both sides of this equation.  Fix some $\alpha\in \S$ with weight $n$ and  let $m=m(\alpha)$.  On the left side of (\ref{eq:scriptE/E}) this $\alpha$ contributes 
	$$\E(z^nt^m) = z^n(1+ zt + \cdots + (zt)^m).$$   
Now consider $\alpha$'s contribution on the right. As $E(\alpha) = \makeset{\alpha+(1^c)}{0<c\leq m}$ we see that $\alpha$ contributes the term $z^n$ and the terms $z^n(zt+(zt)^2+\cdots + (zt)^m)$, respectively.  This proves our first claim.  
	
Next we prove (\ref{eq:N}).  By the uniqueness clauses in Lemma~\ref{lem:N}, it follows that (\ref{eq:N}) is equivalent to  
$$\N F_{\S}(z,t) = \sum_{\alpha\in \S} F_{M(\alpha)}(z,1) + \sum_{\alpha\in \S} F_{N(\alpha)}(z,t).$$
Again consider the contribution of a single partition $\alpha\in\S$, with weight $n$ and $m = m(\alpha)>0$ to both sides of this equation.  On the left, $\alpha$ contributes 
$$\N(z^nt^m) =\left(\frac{1}{1-z^m}\right)\left(\frac{1-(tz)^m}{1-tz}\right)z^n =(1+z^m+z^{2m}+\cdots)(1+ tz + \cdots + (tx)^{m-1})z^n.$$
On the right side, we see that $\alpha$ contributes, via $F_{M(\alpha)}(z,1)$, the terms 
$$(1+z^m+z^{2m}+\cdots)z^n$$ 
since $M(\alpha) = \makeset{\alpha+ (w^m)}{w\geq 0}$.  Additionally, $\alpha$ contributes, via  $F_{N(\alpha)}(z,t)$, the terms
$$(1+z^m+z^{2m}+\cdots)(tz + \cdots + (tz)^{m-1})z^n$$
as $N(\alpha) = \makeset{\alpha+ (w^m) + (1^c)}{w\geq 0\textrm{ and } 0<c<m}$.  This proves (\ref{eq:scriptN/N}).  
\end{proof}

We now turn to the proof of our first theorem.  

\begin{proof}[Proof of Theorem~\ref{thm:F}]
We proceed by induction on $|\mu|$ so that our base case is $\mu = (2)$ with $\tau = (1)$.  In this case $\Theta_\mu = \emptyset$, $\Av(\tau) = \emptyset$, and $\Q(\tau,\mu) = \makeset{(1^c)}{c>0}$.  So $F_\tau = 0$ and 
$$\Theta_\mu\left(\frac{zt}{1-zt}\right) = \left(\frac{zt}{1-zt}\right) = F_{\tau,\mu}(z,t).$$

Now consider some super-strict partition $\mu$ with $|\mu|>2$.  We entertain two cases depending on the difference between the first two parts of $\mu=\tau +(1)$. 

\medskip\noindent
\textbf{Case:} $\mu_1 \geq \mu_2+3$  
\medskip

In this case we can write $\mu = \rho+(2)$ and $\tau = \rho+(1)$ for some super-strict partition $\rho$.  So $\Theta_\mu = \Theta_{\rho+(1)}\cdot (\E)$ where  $\cdot$ is concatenation.  Computing we now have
\begin{align*}
\Theta_\mu\left(\frac{zt}{1-zt}\right) &= \E\left(F_{\rho}(z,1) + F_{\rho,\rho+(1)}(z,t)\right)	&&\textrm{(induction)}\\
&= F_{\rho}(z,1) + \E F_{\rho,\rho+(1)}(z,t),	\\
&=F_{\rho}(z,1) + F_{\rho,\rho+(1)}(z,1) + F_{\rho+(1),\rho+(2)}(z,t) &&\textrm{(Lemmas~\ref{lem:E} and \ref{lem:GF and E/N})}\\
& = F_{\tau}(z,1) + F_{\tau,\mu}(z,t),
\end{align*}
where the last equality follows since $\Av(\rho)\sqcup \Q(\rho,\rho+(1)) = \Av(\rho+(1))$.

\medskip\noindent
\textbf{Case:} $\mu_1 =\mu_2+2$ 
\medskip

As $|\mu|>2$ we have $\mu_2>0$. Set $\rho = (\mu_2,\mu_3,\ldots)$ so that  $\rho+(1)$ has weight at least 2.  Hence $\rho+(1)$ is a super-strict partition.  In terms of $\rho$ we have
$$\mu = \ne\rho +(1)\quad \textrm{and}\quad \tau = \ne\rho.$$ 
Now $\Theta_\mu = \Theta_{\rho+(1)}\cdot (\N)$.   Computing we now have
\begin{align*}
\Theta_\mu\left(\frac{zt}{1-zt}\right) &= \N\left(F_{\rho}(z,1) + F_{\rho,\rho+(1)}(z,t)\right)	&&\textrm{(induction)}\\
&= F_{\rho}(z,1) + \N F_{\rho,\rho+(1)}(z,t),	\\
&=F_{\rho}(z,1) + F_{\rho,\ne\rho}(z,1) + F_{\ne\rho,\ne\rho+(1)}(z,t) &&\textrm{(by Lemmas~\ref{lem:N} and \ref{lem:GF and E/N})}\\
& = F_{\tau}(z,1) + F_{\tau,\mu}(z,t). &&  \qedhere \end{align*}  

\end{proof}

We now turn our attention to showing that $F_\mu(z,1)$ is rational when $\mu$ is super-strict. We begin with a few definitions.

\begin{definition}
We say a bivariate generating function $F(z,t)$ is \textbf{nice} provided that $F(z,z^k)$ is rational for all $k\geq 0$.  Further, we say $F(z,t)$ is \textbf{very nice} provided that $F(z,t)$ is nice and $F(z,0)$ is also rational.  
\end{definition}

\begin{lemma}\label{lem:rho}
If $F(z,t)$ is nice, then $\E F(z,t)$ is very nice.  
\end{lemma}
\begin{proof}
Assume $F(z,t)$ is nice.  By definition of $\E$ we have 
$$ \E F(z,t) = \frac{F(z,1)-ztF(z,zt)}{1-zt}.$$
From this and our assumption about $F$ it follows that $\E F(z,z^k)$, for $k\geq 0$, and $\E F(z,0) = F(z,1)$ is rational. Hence $\E F(z,t)$ is very nice. 
\end{proof}

\begin{lemma}\label{lem:upsilon}
If $F(z,t)$ is very nice, then $\N F(z,t)$ is nice.  
\end{lemma}

\begin{proof}
Let
$$F(z,t) = \sum_{n\geq 1}\sum_{m\geq 0} a_{n,m} z^nt^m.$$
By definition of $\N$ we have, for $m\geq 0$, that
\begin{align*}
\N F(z,z^k) &= \sum_{n\geq 1}\sum_{m\geq 1} a_{n,m} \left(\frac{1}{1-z^m}\right)\left(\frac{1-z^{m(k+1)}}{1-z^{k+1}}\right)z^n + F(z,0)\\
&=\frac{1}{1-z^{k+1}} \sum_{n\geq 1}\sum_{m\geq 1} a_{n,m} (1+z^m + z^{2m} + \cdots + z^{km})z^n + F(z,0)\\
&= \frac{1}{1-z^{k+1}} \left(H(z,1) +H(z,z)+ \cdots+ H(z,z^k)\right)  + F(z,0),
\end{align*}
where $H(z,t):= F(z,t) - F(z,0)$.  
As $F(z,t)$ is very nice, our claim immediately follows.  
\end{proof}

The proof of the next lemma follows immediately from the previous two lemmas.  As such we omit a formal proof.   
\begin{lemma}\label{lem:no consecutive N's}
Consider an arbitrary $F(z,t)$ that is very nice.  Let $\Theta$ be a sequence of the operators $\E$ and $\N$ so that $\Theta$ contains no consecutive $\N$'s.  Then 
$\Theta F(z,1)$ is rational.  
\end{lemma}


\begin{theorem}\label{thm:rational}
	If $\mu$ is super-strict, then $F_\mu(z,1)$ is rational.  
\end{theorem}

\begin{proof}
	If $\mu=(1)$ then $F_\mu(z,1) = 0$ which is rational.  Now assume $\mu$ has weight at least 2.  As $\mu$ is super-strict then $\Theta_\mu$ does not contain consecutive $\N$'s.  As $\mu$ has weight at least 2, our claim follows by Theorem~\ref{thm:F}, the fact that $\frac{zt}{1-zt}$ is very nice, and Lemma~\ref{lem:no consecutive N's}.  
\end{proof}

\begin{corollary}\label{cor:K}
	Fix some $K\geq 1$ and define $S$ to be the set of partitions $\alpha$ such that $|\alpha_i - \alpha_j| \leq K$ for all $i,j\geq 0$.  Then $F_S(z,1)$ is rational.
\end{corollary}
\begin{proof}
	Consider the partition $\mu = (K+2,1)$. As $K\geq 1$ we see that $\mu$ is super-strict.     A simple check shows that $S = \Av(\mu)$.   It now follows immediately by Theorem~\ref{thm:rational} that $F_S(z,1)$ is rational.  
\end{proof}

We end this section with an explanation of how one can explicitly determine the rational function $F_\mu(z,1)$ for a specific super-strict partition $\mu$.  To illustrate, consider the example when $\mu = (5,1)$.  Here $\Theta_\mu = (\N,\E,\E)$ and define
$$F_0(z,t) = \dfrac{zt}{1-zt},\quad F_1 = \N F_0(z,t), \quad F_2 = \E F_1(z,t), \quad F_3 = \E F_2(z,t),$$
By Theorem~\ref{thm:F} we have $F_\mu(z,1) = F_3(z,1)$ and by definition of $\E$ we see that 
$$F_3(z,1) = \left.\E F_2(z,t)\right|_{t=1} = \dfrac{F_2(z,1) - zF_2(z,z) }{1-z}.$$

So to compute $F_3(z,1)$ we must compute $F_2(z,1)$ and $F_2(z,z)$.  Continuing in this manner, we obtain the following dependencies illustrated in the tree below. Note that the proof of Lemma~\ref{lem:no consecutive N's}, yields the dependencies corresponding to $\N$.   We omit references to $F_0(z,0)$ as this is 0.  

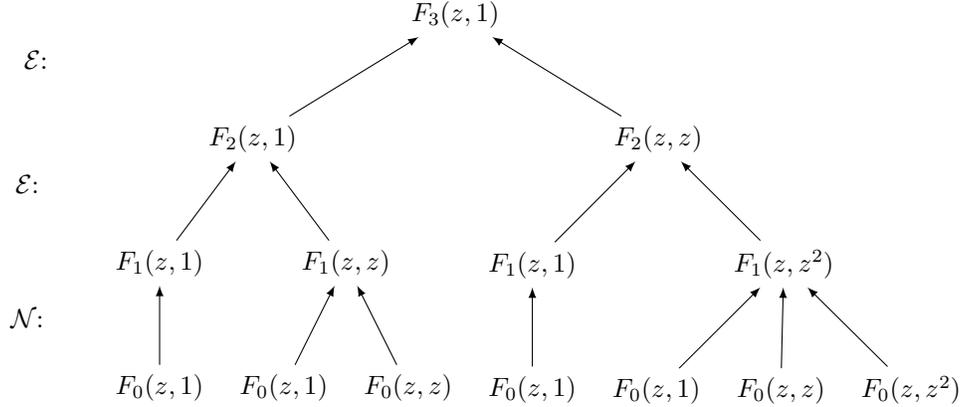
\begin{center}
\begin{forest}
for tree={l sep=30pt, edge={<-,>=latex}}
[{$F_3(z,1)$}
    [{$F_2(z,1)$},  edge label={node[xshift = -8.2em, yshift= 3em] {$\E$:}}  
      [{$F_1(z,1)$},edge label={node[xshift = -5em, yshift= 3em] {$\E$:}} 
      		[{$F_0(z,1)$}, edge label={node[xshift = -5em, yshift= 2.6em] {$\N$:}}
      		]
      ] 
      [{$F_1(z,z)$}
      		[{$F_0(z,1)$}]
      		[{$F_0(z,z)$}]
      ] 
	]
    [{$F_2(z,z)$}
      [{$F_1(z,1)$}
      		[{$F_0(z,1)$}]
      ] 
      [{$F_1(z,z^2)$}
      		[{$F_0(z,1)$}]
      		[{$F_0(z,z)$}]
      		[{$F_0(z,z^2)$}]
      ] 
  ] 
]
\end{forest}
\end{center}
In this fashion we see that $F_\mu(z,1)$ may be computed for any super-strict partition $\mu$.

\begin{remark}\label{rmk:metacyclic}
The pattern $\mu = (5,2)$ has a surprising connection to group theory.  We recall that a group $G$ is said to be a metacyclic if there exists a cyclic normal subgroup $N$ such that $G/N$ is cyclic.  In \cite{Liedahl:Enume1996}, Liedahl enumerates metacyclic $p$-groups and proves that for any odd prime $p$ the number of such groups of order $p^n$ is given by the generating function
$$G(z) = \frac{-z(z^7 - 2z^5 + z^3 + z^2 - z - 1)}{(z - 1)^4(z + 1)^2(z^2 + z + 1)}.$$
Coincidentally, the generating function $F_{(5,2)}(z,1)$, which was computed using the  recursive algorithm described above, is equal to $G(z)$.   We do not know of a bijective proof of this fact.
\end{remark}

\subsection{Proofs of Lemmas~\ref{lem:E} and \ref{lem:N}}\label{sec:lemmas 1 and 2}

To prove Lemmas~\ref{lem:E} and \ref{lem:N} we first consolidate some basic facts about partition containment.   

\begin{lemma}\label{lem:list of prop}
Let $\alpha,\mu$ be partitions with $\mu$ strict.  Let $\alpha^-$ be any partition that can be obtained by deleting the top $k<m(\alpha)$ rows from $\alpha$.  Fixing $a>\alpha_1$ and $m>0$ set
$$\alpha^+ = (a^m, \alpha_1,\alpha_2,\ldots).$$
We have the following:
\begin{enumerate}[label=\roman*)]
	\item $\alpha$  contains $\mu$ if and only if $\alpha^-$ contains $\mu$.
	\item $\alpha$ contains $\mu$ if and only if $\alpha^+$ contains $\ne\mu$.
	\item $\alpha$ contains $\mu$ if and only if $\alpha+(1^c)$ contains $\mu+(1)$ where $0<c\leq m(\alpha)$.  
\end{enumerate}
\end{lemma}

\begin{proof}
The reverse direction of \emph{i)} and the forward direction of \emph{ii)} are immediate.  The remaining directions follow from the following fact.  If $\alpha$ contains a strict partition $\mu$ then the rows not deleted from $\alpha$ to obtain $\mu$ must be of distinct length.  

For the proof of \emph{iii)} the reverse direction is immediate.  For the forward direction assume the set of rows and columns deleted to obtain $\mu$ are $R$ and $C$, respectively.  First note that we may assume $1\notin R$ as otherwise set
$$R'= (R \setminus \{1\}) \cup \{i\} \quad\textrm{ and }\quad C'=C\cup \{\alpha_i+1,\alpha_i+2,\ldots, \alpha_1\}$$
where $i$ was the least index not already included in $R$. 
Observe that deleting the rows in $R'$ and columns in $C'$ will also yield $\mu$.  This, combined with the fact mentioned in the previous paragraph, implies that $\alpha+(1^c)$ contains $\mu+(1)$.
\end{proof}

In what follows we make frequent use of this lemma.  To facilitate readability we only reference a given statement within this lemma by its number with no explicit reference to the lemma.  

\begin{proof}[Proof of Lemma~\ref{lem:E}]
For any $0<c\leq m(\alpha)$ it follows from \emph{iii)} that
$$\alpha\in \Q(\mu,\mu +(1))\quad \Longleftrightarrow\quad  \alpha +(1^c)\in \Q(\mu+(1),\mu+(2)).$$
In particular this implies that $E\Q(\mu,\mu+(1)) \subseteq \Q(\mu+(1),\mu+(2))$.  As every partition in $\Q(\mu+(1),\mu+(2))$ has at least two columns (as such partitions contain $\mu+(1)$) we can write any element in $\Q(\mu+(1),\mu+(2))$ as $\alpha+(1^c)$ for some partition $\alpha$ with $0<c\leq m(\alpha)$.  This together with our first observation yields the reverse inclusion.  

For the uniqueness claim, fix $\beta\in E\Q(\mu,\mu+(1))$ and assume $\alpha, \eta \in \Q(\mu,\mu+(1))$ such that 
$$\alpha +(1^c) = \beta = \eta +(1^d)$$
where $c\leq m(\alpha)$ and $d\leq m(\eta )$.   As $c,d>0$ it follows that the  rightmost column in $\beta$ has length $c=d$.  So, $\alpha = \beta$.  This completes our proof.
\end{proof}

\begin{lemma}\label{lem:Q(mu,nemu)}
	Let $\mu$ be strict and fix $\alpha \in \Q(\mu,\ne\mu)$.  Then  $\alpha$ decomposes uniquely as 
$$\alpha = (b_1^m,b_2,b_3,\ldots) +(w^m)$$
where $w\geq 0$, $b_1>b_2$, and $(b_1^m,b_2,b_3,\ldots) \in \Q(\mu, \mu+(1))$.  
\end{lemma}

\begin{example}
Take $\mu = (3,1)$ so $\ne\mu = (4,3,1)$ and $\alpha = (7,7,7,4,3,3)$.  Note that $\alpha \in \Q(\mu,\ne\mu)$ since $\alpha$ contains $\mu$, but not $\ne\mu$.  In this case $\alpha$ decomposes as $\alpha = (5,5,5,4,3,3)+(2,2,2)= (5^3,4,3,3)+(2^3)$, where $(5^3,4,3,3) \in \Q(\mu,\mu + (1))$.
\end{example}

\begin{proof}[Proof of Lemma~\ref{lem:Q(mu,nemu)}]
	Set $\alpha = (a_1^m, a_2,\ldots) \in \Q(\mu,\ne\mu)$ so that $a_1>a_2$ and $m>0$.  
Starting with the fact that $\alpha$ contains $\mu$ define $i$ to be the smallest column index so that if $\beta$ is the partition obtained by deleting columns $i+1,i+2,\ldots$ then $\beta$ contains $\mu$.  It is immediate from our choice of $i$ that $\beta$ avoids $\mu+(1)$.  For the existence claim, it now suffices to prove that $i>a_2$ as then we have 
$$\alpha = \underbrace{(i^m, a_2,a_3\ldots )}_{\alpha'} + (a_1-i)^m.$$ 
For a contradiction, assume $i\leq a_2$.  So
$$\alpha' = (a_2^{m+1}, a_3,\ldots),$$
the partition obtained by deleting all columns of length $m$ from $\alpha$, contains $\mu$.  By \emph{i)} it follows that $(a_2, a_3,\ldots)$ contains $\mu$. As $a_1>a_2$, it follows from \emph{ii)} that $\alpha = (a_1^m,a_2,a_3,\ldots)$ contains $\ne\mu$.   This contradicts our choice of $\alpha$. 

To prove uniqueness assume for a contradiction that we can write 
$$\underbrace{(c_1^m, c_2,c_3 \ldots)}_\eta  + (w^m) = \alpha = \underbrace{(d_1^m, c_2, c_3, \ldots)}_\delta + (u^m),$$
where $d_1>c_1>c_2$ and $\eta ,\delta\in \Q(\mu,\mu+(1))$.  As $\eta $ contains $\mu$ we see by \emph{iii)} that $\eta +(1^m)$ contains $\mu+(1)$.  As $d_1>c_1$, it now follows that $\delta$ contains $\mu+(1)$, which it does not. This establishes our uniqueness clause and completes our proof.
\end{proof}

Equipped with the above lemmas we now prove Lemma~\ref{lem:N}.

\begin{proof}[Proof of Lemma~\ref{lem:N}]
Throughout fix a partition $\alpha = (a_1^m, a_2, a_3,\ldots)$  with $a_1>a_2$ so that $m(\alpha) = m$.  We first prove that
$$M\Q(\mu,\mu+(1)) =  \Q(\mu,\ne\mu).$$
To this end take some $\alpha\in \Q(\mu,\mu+(1))$ and $w\geq 0$ so that $\alpha + (w^m)\in M\Q(\mu,\mu+(1))$.  First, as $\alpha$ contains $\mu$ then $\alpha+(w^m)$ also contains $\mu$.   Next, as $\alpha$ avoids $\mu+(1)$ then $\alpha$ also avoids $\ne\mu$.  By \emph{ii)} it follows that $(a_2,a_3,\ldots)$ avoids $\mu$.  Another application of \emph{ii)}  implies that  $\alpha+(w^m) = ( (a_1+w)^m, a_2,a_3,\ldots)$ avoids $\ne\mu$. We conclude that $M\Q(\mu,\mu+(1))\subseteq Q(\mu,\ne\mu)$.  

The reverse inclusion, as well as our uniqueness claim in this case, follows directly from Lemma~\ref{lem:Q(mu,nemu)}.

\medskip

Next we turn our attention to the proof that 
$$N\Q(\mu,\mu+(1)) =  \Q(\ne\mu,\ne\mu+(1)).$$
By the first part of this proof we see that
\begin{align*}
    N\Q(\mu,\mu+(1)) &=\makeset{\alpha+(1^c)}{\alpha\in M\Q(\mu,\mu+(1)),\ 0<c<m(\alpha)}\\
    &= \makeset{\alpha+(1^c)}{\alpha\in \Q(\mu,\ne\mu),\  0<c<m(\alpha)}.
\end{align*}
By \emph{iii)} observe that if $\alpha$ avoids $\ne\mu$ then $\alpha+(1^c)$ avoids $\ne\mu +(1)$.  

Now assume $\alpha$ contains $\mu$.  By \emph{i)} and the fact that $c<m$ we see that $(a_1^{m-c}, a_2,a_3,\ldots)$
contains $\mu$.  As $c>0$ it follows from \emph{ii)} that 
$$\alpha+(1^c) = ( (a_1+1)^c, a_1^{m-c}, a_2,\ldots)$$
contains $\ne\mu$.  It now follows that $N\Q(\mu,\mu+(1)) \subseteq \Q(\ne\mu, \ne\mu+(1))$.  
 
 To establish the reverse inclusion consider some $\beta\in \Q(\ne\mu,\ne\mu+(1))$ and set $c = m(\beta)$.  Let $\alpha$ be the partition such that 
 $$\beta = \alpha+(1^c).$$  
 As $\beta$ avoids $\ne\mu+(1)$ it follows by \emph{iii)} that $\alpha$ avoids $\ne\mu$.   Additionally, as $\beta$ contains $\ne\mu$ it follows by \emph{ii)} that the result of deleting the top $c$ rows from $\beta$ contain $\mu$.  Hence $\alpha$ contains $\mu$ and so $\alpha\in \Q(\mu,\ne\mu)$.  It remains to show  $c<m(\alpha)$.  Clearly $c\leq m(\alpha)$ so for a contradiction assume we have equality. This means  we can write
 $$\beta = (b_1^c, b_2,\ldots)$$
 where $b_1> 1+ b_2$.  As $\beta$ contains $\ne\mu$ it follows by \emph{i)} and then \emph{ii)} that $((b_2+1)^c,b_2,\ldots)$ contains $\ne\mu$.  As $b_1> b_2+1$ it follows by \emph{iii)} that $\beta$ contains $\ne\mu+(1)$, our desired contradiction.  The reverse inclusion now follows. 

To establish uniqueness in this case, fix $\beta \in N\Q(\mu, \mu +(1))$ and let $\alpha, \eta \in Q(\mu,\mu+(1))$ such that $\beta \in N(\alpha)\cap N(\eta )$.  If $\beta'$ is the partition obtained by deleting the leftmost column of $\beta$, then we see that $\beta' \in M(\alpha)\cap M(\eta )$.  By the uniqueness of the previous case, we conclude that $\alpha = \eta $.  
\end{proof}

\section{Asymptotics of $\Av_n(\mu)$}\label{sec:growth}


In this section we obtain asymptotics for the growth of the sequence $|\Av_1(\mu)|,|\Av_2(\mu)|,|\Av_3(\mu)|,\ldots$ for any strict partition $\mu$. Recalling Theorem~\ref{thm:strict} we remind the reader that no loss in generality results by considering only strict partitions.   

Throughout we use Vinogradov's notation $f \ll_a g $ to mean $f = O_a(g)$, with variables in the subscript indicating possible dependence of the implied constant on those variables.  We need the following facts about the function $\sigma_k(n) = \sum_{d|n} d^k$ (see \cite{HardyWright}):

\begin{align}
    \sum_{m=1}^n \sigma_0(n) &= n \log n + (2\gamma -1)n + O(n^\theta) \label{eq:sigma0}, \\ 
    \sum_{m=1}^n \sigma_1(n) &= \frac{\zeta(2)n^2}{2} + O(n\log n), \label{eq:sigma1}\\
    \sum_{m=1}^n \sigma_k(n) &= \frac{\zeta(k+1)n^{k+1}}{k+1} + O(n^k) \hspace{1cm} (k>1). \label{eq:sigmak}
\end{align}
In \eqref{eq:sigma0} $\gamma$ is the Euler–Mascheroni constant and the value of $\theta$ is the subject of the Dirichlet divisor problem.    At present it is known, due to Huxley \cite{Huxley} that we can take $\theta=131/416$.  We also recall that \begin{equation}
    \sum_{m=1}^n m^a \log^b m = \frac{m^{a+1} \log^b m}{a+1} + O\left(\frac{b}{m+1} m^{a+1}\log^{b-1} m \right). \label{eq:logsum}
\end{equation}
for any $a,b>0$. One can similarly derive that \begin{equation}
    \sum_{m=1}^n \sigma_a(m) \log^b m = \frac{\zeta(a+1)m^{a+1} \log^b m}{a+1} + O\left(\frac{b}{m+1} m^{a+1}\log^{b-1}m \right). \label{eq:sigmalogsum}
\end{equation}

We first consider a few sporadic cases when $\mu$ has small weight which are also listed in Table \ref{tab:my_label} at the end of the paper.
\begin{theorem} \label{thm:sporadic}
We have  
\end{theorem} 
\vspace{-8mm}  
\begin{align*}
    |\Av_n \big((1)\big)| = 0,  \hspace{1cm} & \hspace{1cm} |\Av_n \big((2)\big)| = 1, \\
    |\Av_n \big((2,1)\big)| = \sigma_0(n), \hspace{0.35cm} & \hspace{1cm}
    |\Av_n \big((3)\big)| = \left \lfloor \frac{n}{2}\right\rfloor = \frac{n}{2} + O(1), \\
    |\Av_n \big((3,1)\big)| = n, \hspace{1cm} & \hspace{1cm} 
    |\Av_n \big((3,2)\big)| = n\log n + (2\gamma -2)n + O\left(n^{131/416}\right).
\end{align*}
\begin{proof}
Clearly every partition contains $(1)$.  When $\mu = (2)$ we find that $\Av_n(\mu)=\makeset{(1^n)}{n\geq 1}$ and so $|\Av_n(\mu)| = 1$.  The Ferrers board of a partition avoiding $(2,1)$ must be a rectangle, and so we obtain one partition for each divisor of $n$, thus $|\Av_n \big((2,1)\big)| = \sigma_0(n)$.

The partitions avoiding $(3)$ are precisely those with at most two columns and hence $$|\Av_n\big((3)\big)| = \left \lfloor \frac{n}{2}\right\rfloor = \frac{n}{2} + O(1).$$  
The partitions in $\Av_n((3,1))$ are precisely those whose Ferrers board consisting of $x$ columns of height $y$, for some $x,y>0$ together a single column of height $0\leq r< y$.  So  $n= xy + r$.  By the division algorithm we get one such decomposition of $n$ for each $0<y\leq n$.  Hence  $$|\Av_n\big((3,1)\big)| = n.$$  

Lastly the partitions avoiding $(3,2)$ of weight $n$ are those of the form $\makeset{(a^b,1^c)}{a\cdot b+c =n}$.  In other words, these are partitions obtained by taking a rectangle of weight $m:=a\cdot b$  and then extending the first column by $c = n-m$.  Counting these we have
$$|\Av_n\big((3,2)\big)|=  1+ \sum_{m=1}^n (\sigma_0(m) - 1) = n\log n + (2\gamma -2)n + O(n^{131/416}),$$
where we subtract 1 in the second step to avoid over counting the partition $(1^n)$ and the second equality is obtained by (\ref{eq:sigma0}).
\end{proof} 

We now introduce a definition that will be used extensively throughout the remainder of this section.

\begin{definition}
Let $\alpha$ be a partition.  Then its Ferrers board is obtained by horizontally concatenating $k>0$ rectangles of widths $x_1,\ldots, x_k>0$ and strictly decreasing heights $y_1> y_2> \cdots> y_k>0$ so that
$$|\alpha|=x_1y_1 + \cdots +x_ky_k.$$
We call $k$ the number of \emph{distinct magnitudes} of $\alpha$ and the pair of height and width sequences the \emph{rectangular decomposition} of $\alpha$. 
\end{definition}

\begin{example}
The partition 
$$\ydiagram[*(gray)]{2,2,2,2,2,2,2}*{2+3,2+3,2+3,2+3,2+3,}*[*(gray)]{5+1, 5+1,5+1}$$ 
has the unique rectangular decomposition of $2\times 7 + 3\times 5 + 1\times 3$.  
\end{example}

A few notes about this definition before continuing.  As we insist that the sequence of heights is strictly decreasing it follows that a rectangular decomposition is unique.  If we write a partition as $\alpha = (a_1^{e_1}, \ldots, a_k^{e_k})$ where $a_1>a_2>\cdots >a_k>0$ and $e_i>0$ then traditionally $a_1,a_2,\ldots, a_k$ are referred to as the $k$-distinct magnitudes of $\alpha$.  An easy check confirms that the number of distinct magnitudes in our definition above is the same as the number of distinct magnitudes in this traditional sense.  Lastly, although a rectangular decomposition is defined to be a pair of sequences we abuse notation and refer to a sum of the form displayed above as a rectangular decomposition.  In this case the $x_i$'s and $y_i$'s shall always denote widths and heights, respectively.

Next observe that if $\mu$ is strict then no partition in $\Av(\mu)$ may have more than $\mu_1-1$ distinct magnitudes.  This follows for if some $\alpha\in \Av(\mu)$ had at least $\mu_1$ distinct magnitudes, then $\alpha$ would contain the staircase $(\mu_1,\mu_1-1,\ldots, 1)$.  As $\mu$ has distinct parts, it follows that $\alpha$ would contain $\mu$.  This motivates the next definition.  

\begin{definition}
For a strict partition $\mu$ and $n>0$ we denote by $D_n(\mu)$ the set of partitions in $\Av_n(\mu)$ having exactly $\mu_1-1$ distinct parts.  For completeness we set $D_0(\mu) = \emptyset$.  We also define $D(\mu) = \cup_{n=0}^\infty D_n(\mu)$.  
\end{definition}

One of the key ideas going forward is to partition the set $\Av(\mu)$ into those partitions with exactly $\mu_1-1$ distinct magnitudes and those with fewer distinct magnitudes.  We will see that the contribution from partitions in $\Av_n(\mu)\setminus D_n(\mu)$ is asymptotically negligible.  

The next lemma shows that the rectangular decomposition of partitions of $D(\mu)$ has a nice characterization.  As our proof uses Lemma~\ref{lem:list of prop} the reader is encouraged to review its statement.  Further as this lemma contains several parts we shall, for example, simply write ``by \emph{ii)}'' instead of the more verbose ``by part \emph{ii)} of Lemma~\ref{lem:list of prop}'' each time.   

\begin{lemma}\label{lem:rd of mu}
Let $\mu$ be a strict partition with $\mu_1\geq 2$.  Then the set $D(\mu)$ consists precisely of those partitions $\alpha$ whose rectangular decomposition $$x_1y_1+\cdots + x_{\mu_1-1}y_{\mu_1-1}$$
satisfies the restriction $x_i=1$ whenever $\mu$ does not have a part of size $i$.
\end{lemma}

\begin{proof}
	We prove this by induction on $|\mu|$.  As $\mu_1\geq 2$ our base case is when $\mu= (2)$.  The set $\Av\big((2)\big)$ consists of all partitions of the form $(1^e)$ for $e>0$ which are precisely the partitions whose rectangular decomposition is $1\cdot y_1$ for some $y_1>0$ as claimed.
	
Now consider a strict $\mu$ with $|\mu|>2$. Denote by $R(\mu)$ the collection of all partitions whose rectangular decomposition is as in the statement of the lemma.   We consider two cases depending on the size of $\mu_1-\mu_2$. 

\medskip\noindent 
\textbf{Case 1:} $\mu_1-\mu_2>1$
\medskip

Set $\tau = (\mu_1-1, \mu_2,\mu_3,\ldots)$ so that $\mu = \tau +(1)$.  By induction $R(\tau) = D(\tau)$. Consider some $\beta \in R(\mu)$ with the aim of showing that $\beta\in D(\mu)$.  By definition $\beta$'s rectangular decomposition is 
$$|\beta|=x_1y_1 + \cdots + x_{\mu_1-2}y_{\mu_1-2}+ 1\cdot y_{\mu_1-1},$$
where $y_1>y_2>\cdots > y_{\mu_1-1}$ and $x_i=1$ whenever $\mu$ does not have the part of size $i$.  Let $\alpha$ be the partition whose rectangular decomposition is 
$$|\alpha|=x_1y_1 + \cdots + x_{\mu_1-2}y_{\mu_1-2}.$$
Clearly $\alpha \in R(\tau) = D(\tau)$.  So $\alpha$ avoids $\tau$ and has $\mu_1-2$ distinct magnitudes.  As $\beta = \alpha + (1^c)$ for $c = y_{\mu_1-1}$ we know by  \emph{iii)} that $\beta$ avoids $\mu = \tau +(1)$.  As $\beta$ has $\mu_1-1$ distinct magnitudes then $\beta \in D(\mu)$ and we conclude that $R(\mu)\subseteq D(\mu)$.  

For the reverse inclusion consider $\beta\in D(\mu)$ and write it as 
$\beta = \alpha + (1^c)$
where $c\leq$ the length of the rightmost column in $\alpha$.  As $\beta$ avoids $\mu$,  \emph{iii)} implies that $\alpha$ avoids $\tau$.  As $\beta$ has $\mu_1-1$ distinct magnitudes it follows that $\alpha$ must have either $\mu_1-1$ or $\mu_1-2$ distinct magnitudes.  If $\alpha$ had $\mu_1-1$ distinct magnitudes, then it would contain the staircase $(\mu_1-1, \mu_1-2,\ldots, 1)$ in which case it would contain the strict partition $\tau$.  So $\alpha$ has $\mu_1-2 = \tau_1-1$ distinct magnitudes and hence $\alpha \in D(\tau)=R(\tau)$.  
So $\alpha$ has the rectangular decomposition
$$|\alpha| = x_1y_1 + \cdots + x_{\mu_1-2}y_{\mu_1-2},$$
where $x_i = 1$ whenever $\tau$ does not have a part of size $i$.  Consequently the rectangular decomposition for $\beta$ is
$$|\beta| = x_1y_1 + \cdots + x_{\mu_1-2}y_{\mu_1-2}+ 1\cdot c,$$
with $c<y_{\mu_1-2}$ as $\beta$ has exactly one more distinct magnitude than $\alpha$.  Hence $\beta \in R(\mu)$ proving that $D(\mu)\subseteq R(\mu)$.  

\medskip\noindent 
\textbf{Case 2:} $\mu_1-\mu_2 = 1$ 
\medskip

Let $\tau = (\mu_1, \mu_3,\mu_4,\ldots)$, i.e., the partition obtained be removing $\mu_2$ from $\mu$.    As $\mu_1-\mu_2 =1$ and $\mu_1\geq 2$ then $\mu_2>0$.  So $|\tau|< |\mu|$, and hence by induction, $R(\tau) = D(\tau)$.  Now consider some $\beta\in R(\mu)$ with the aim of showing that $\beta\in D(\mu)$.  By definition $\beta$'s rectangular decomposition is 
$$|\beta|=x_1y_1 + \cdots + x_{\mu_1-2}y_{\mu_1-2}+ x_{\mu_1-1} y_{\mu_1-1}$$
where $y_1>y_2>\cdots > y_{\mu_1-1}$ and $x_i=1$ whenever $\mu$ does not have the part of size $i$.  As $\beta$ has $\mu_1-1$ distinct parts it suffices to show that $\beta$ avoids $\mu$. 

For a contradiction assume $\beta$ contains $\mu$. Define $\eta $ to be the partition obtained by deleting the top $y_{\mu_1-1}$ rows from $\beta$.   Now if $\beta$ contains $\mu$ then it follows by \emph{ii)} and the fact that $\mu_1 = \mu_2+1$ that $\eta $ contains $(\mu_2,\mu_3,\ldots)$.  Next define $\alpha$ to be the partition whose rectangular decomposition is 
$$|\alpha|=x_1y_1 + \cdots + x_{\mu_1-2}y_{\mu_1-2} + 1\cdot y_{\mu_1-1}.$$
It follows that $\alpha \in R(\tau)=D(\tau)$.   Observe that $\alpha$ is obtained by adding $y_{\mu_1-1}$ rows of length $\eta _1+1$ to $\eta $.  As $\eta $ contains $(\mu_2,\mu_3,\ldots)$ it follows by \emph{ii)} that $\alpha$ contains $\mu=(\mu_2+1,\mu_2, \mu_3,\ldots)$ which, in turn, implies that $\alpha$ contains $\tau$, a contradiction.  We conclude that $R(\mu) \subseteq D(\mu)$.

To obtain the reverse inclusion take some $\alpha\in D(\mu)$.  Setting $k = \mu_1-1 = \mu_2$ write
$$\alpha = (a_1^{e_1}, a_2^{e_2},a_3^{e_3}, \ldots, a_{k}^{e_{k}})$$ 
for some $a_1>a_2>\cdots $ and $e_i>0$.  Now define
$$\eta  = (a_2^{e_1+e_2},a_3^{e_3},\ldots, a_k^{e_k})\quad\text{and}\quad  \beta = \eta  +(1^{e_1}).$$
As $\alpha$ avoids $\mu$ it follows from \emph{ii)} and then \emph{i)} that $\eta $ avoids $(\mu_2,\mu_3,\ldots)$.  By \emph{iii)} we further see that $\beta$ avoids $\tau = (\mu_2+1,\mu_3,\ldots)$.  We conclude that $\beta\in D(\tau) = R(\tau)$.  A straightforward check now shows that as $\beta\in R(\tau)$ then $\alpha\in R(\mu)$.  This proves the reverse inclusion.  
\end{proof}

In light of this lemma we make the following definition.  

\begin{definition}
Take $\mu$ as in the statement of the previous lemma and fix $\alpha\in D(\mu)$.  Consider the rectangular decomposition of $\alpha$.  If $\mu$ has no part of size $i$ then we say the $i$th rectangle in this decomposition is \emph{thin}.  If $\mu$ has a part of size $i$ then we say the $i$th rectangle is \emph{wide}. Thus thin rectangles are precisely those rectangles whose width is forced to be 1.  
\end{definition}

We conclude this section with an example illustrating these ideas.  

\begin{example}
Consider the pattern $\mu = (4,3,1)$.  This partition is strict, with $\mu_1 = 4$, and so partitions in $D(\mu)$ will have 3 distinct part sizes.  Furthermore, since $\mu$ does not have a part of size 2, the second rectangle in the rectangular decomposition of the partitions in $D(\mu)$ is thin.  Thus partitions in  $D((4,3,1))$ look like
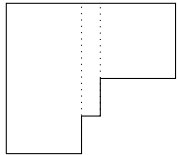
\begin{center}
\begin{tikzpicture}
    \draw[-](0,0) -- (1,0) -- (1,0.5) -- (1.25, 0.5) -- (1.25,1) -- (2.25,1) -- (2.25, 2) -- (0,2) -- (0,0);
    \draw[dotted] (1,0) --(1,2);
    \draw[dotted] (1.25,.5) --(1.25,2);
\end{tikzpicture}
\end{center}
and their rectangular decomposition is of the form $x_1y_1 + 1y_2 + x_3y_3$ with $y_1>y_2>y_3>0$.  Here the first and third rectangles are wide while the second rectangle is  thin.  
\end{example}





\subsection{Avoiding staircase partitions}\label{ssec:staircases}

The goal of this subsection is to give an asymptotic formula for the count $|\Av_n(\mu)|$ in the special case where  $\mu=(k+1,k,k-1,\ldots, 2,1)$.  This special case is closely related to work that has already been considered in the literature.  Notice that if $\alpha \in \Av_n(\mu)$ but not in $D_n(\mu)$, i.e., it has fewer than $k$ distinct magnitudes, then it avoids $(k,k-1,\ldots, 1)$ thus we have
\begin{equation}
    \Av_n(\mu) = D_n(\mu) \cup \Av_n\big((k,k-1,\ldots 1)\big). \label{eqn:stairdecomp}
\end{equation}

By Lemma~\ref{lem:rd of mu} partitions in $D_n(\mu)$ correspond to representations of 
$n = \sum_{i=1}^k x_iy_i $ 
with $x_i>0$ for all $i$ and  $y_1>y_2>\cdots >y_k>0$.  For example if $k = 3$ we get all partitions of the form

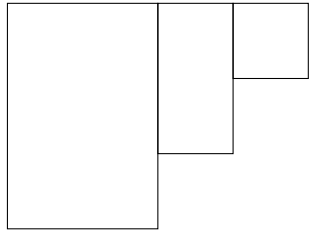
\begin{center}
\begin{tikzpicture}
    \draw (0,0) rectangle (2,3);
    \draw (2,1) rectangle (3,3);
    \draw (3,2) rectangle (4,3);
\end{tikzpicture}\ .
\end{center}
where the width of these three rectangles are arbitrary and their lengths, from left to right, are $y_1>y_2>y_3>0$. In general, partitions of this sort are those whose parts have $k$ distinct magnitudes.  Such partitions have been studied by many authors going back to MacMahon \cite{MacMahon}.  The count of such representations was considered by Andrews \cite{AndrewsSLB} who uses results of Ingham \cite{Ingham}, Estermann \cites{Estermann1,Estermann2} and Johnson \cite{Johnson} on counting the representations of $n$ as a sum of $k$ products of pairs of positive integers.  Their results give estimates for the number $\nu_k(n)$ of representations of $n = \sum_{i=1}^k x_iy_i $ 
where $x_i,y_i>0$ but without any restrictions on the values $y_i$.    In particular they show the following

\begin{numcases}{\nu_k(n) =}
\sigma_0(n) & $k=1$ \label{eq:nk0}\\
\frac{1}{\zeta(2)}\sigma_{1}(n)\log^2 n-\frac{4}{\zeta(2)} n \log n \sigma_{\hspace{-0.7mm}\shortminus 1} '(n) + O\left(\sigma_1(n) \log n\right)\label{eq:nk2-precise} \\ 
\hspace{2.3cm} = \frac{1}{\zeta(2)}\sigma_{1}(n)\log^2 n\left(1+O\left(\frac{\log \log n}{\log n}\right)\right) & $k=2$ \label{eq:nk2} \\
\frac{1}{(k-1)!\zeta(k)}\sigma_{k-1}(n)\log^k n\left(1+O\left(\frac{1}{\log n}\right)\right) & $k\geq 3$ \label{eq:nk3} \ .
\end{numcases}

Observe that when $k=2$ the relative error term $O\left(\frac{1}{\log n}\right)$ in the $k\geq 3$ case must be replaced by $O\left(\frac{\log \log n}{\log n}\right)$ if only the main term is used. In the second term the function $\sigma_{\hspace{-0.7mm}\shortminus 1} '(n) = \left. \sum_{d|n}\frac{\log d}{d} = \frac{d}{ds} \sigma_s(n)\right|_{s=-1}$.  We caution the reader that this result for $k=2$ was stated incorrectly without this error term in the introduction of \cite{Johnson}.  However,  it is derived with this term by Ingham \cite{Ingham} and a more precise result, including further lower order terms is obtained in \cite{Estermann1}.

Going forward, it will be useful to recall bounds for the size of the function $\sigma_j(n)$. For $j>1$ we have \[n^j \leq \sigma_j(n) = n^j \sum_{d|n} \frac{1}{d^j} \leq n^j \zeta(j).\]  In particular, $\sigma_j(n)$ is bounded above and below by a constant times $n^j$.  When $j=1$, the function is not quite so well behaved, however we have the bounds $n \leq \sigma_1(n) \ll n \log \log n.$ Finally, when $j=0$ we have the divisor function, which is much less well behaved, however we can still bound it as $2 \leq \sigma_0(n) = o(n^\epsilon)$ for every $\epsilon>0$.  While $\sigma_0(n)$ is poorly behaved, it behaves well on average, as seen in \eqref{eq:sigma0}.

We now count partitions avoiding a staircase partition $\mu=(k+1,k,k-1,\ldots, 2,1)$ for any $k\geq 1$.

\begin{theorem}\label{thm:staircase}
Fix $k\geq 1$ and let $\mu=(k+1,k,k-1,\ldots, 2,1)$.  Then $|\Av_n(\mu)| = \frac{\nu_k(n)}{k!}\left(1+O\left(\frac{1}{\log n}\right)\right)$ and 
\begin{equation}|\Av_n(\mu)| =\begin{cases}
\sigma_0(n) & k=1  \\
\frac{1}{2\zeta(2)}\sigma_{1}(n)\log^2 n\left(1+O\left(\frac{\log \log n}{\log n}\right)\right) & k=2 \\
\frac{1}{k!(k-1)!\zeta(k)}\sigma_{k-1}(n)\log^k n\left(1+O\left(\frac{1}{\log n}\right)\right) & k\geq 3  \ .
\end{cases} \label{eq:stairasymp} \end{equation}

\end{theorem}

Note that the only difference between the $k=2$ and $k\geq 3$ case in the statement of the theorem is the $\log \log n$ term in the numerator of the error term.

\begin{proof}
When $k=1$, $\mu = (2,1)$ and, as observed in Theorem \ref{thm:sporadic}, the set $\Av_n(\mu)$ consists of partitions with one distinct part or, equivalently, partitions whose Ferrers board is a rectangle so $|\Av_n(\mu)|=\nu_1(n)=\sigma_0(n)$, the number of divisors of $n$.

For $k\geq 2$ we prove this by induction on $k$.  Let $k=2$ and $\mu=(3,2,1)$. As in \eqref{eqn:stairdecomp} we have ${\Av_n(\mu)=D_n(\mu) \cup \Av_n\big((2,1)\big)}$. Since $|\Av_n((2,1))|=\sigma_0(n)=o(n)$ it suffices to count elements of  $D_n(\mu)$.

By Lemma \ref{lem:rd of mu} every partition in $D_n(\mu)$ can be described as a representation of $n=x_1y_1+x_2y_2$ with $x_i>0$ and $y_1>y_2$. Equation \eqref{eq:nk2} gives an expression for $\nu_2(n)$, the number of such representations of $n$ without the restriction $y_1>y_2$. 

To count elements of $D_n(\mu)$ however we must exclude from this count those representations where $y_1=y_2$, which would mean that $n=y_1(x_1+x_2)$.  Note that each such representation of $n$ is obtained uniquely from a factorization of $n$ as $n=de$ by taking $y_1=d$ and $x_1+x_2=e$.   For a fixed value of $e$ there are $e-1$ choices for $x_1$ and $x_2$, so the total number of such representations of $n$ is $\displaystyle{\sum_{e|n} (e-1)} = \sigma_1(n)-\sigma_0(n)$.  Finally, each representation with $y_1 \neq y_2$, occurs twice, since heights $y_i$ can occur permuted in either order, thus using \eqref{eq:nk2} and the fact that $\sigma_1(n) - \sigma_0(n) = O(n\log\log n)$ we have
\begin{align*}
\left|D_n\big((3,2,1)\big)\right| &= \tfrac{1}{2}\big(\nu_2(n)-\sigma_1(n)+\sigma_0(n)\big)\\
&= \frac{1}{2\zeta(2)}\sigma_1(n)\log^2 n + O\left(\sigma_1(n)\log n \log \log n \right)\\
&= \frac{1}{2\zeta(2)}\sigma_1(n)\log^2 n \left(1+O\left(\frac{\log \log n}{\log n}\right)\right). \nonumber
\end{align*}

Now suppose $k>2$ and that the result holds for all smaller values of $k$.  Again we write
$$\Av_n(\mu) = D_n(\mu) \cup \Av_n\big((k,k-1,\ldots 1)\big).$$
By our induction hypothesis, the size of $\Av_n\big((k,k-1,\ldots, 1)\big)$ is smaller than the error term in our desired result.  Therefore we need only count the partitions in $D_n(\mu)$.  By Lemma \ref{lem:rd of mu} the partitions in this set are precisely the set of partitions that admit the rectangular representations $n=x_1y_1+ \cdots + x_ky_k$ where $y_1> \cdots >y_k>0$.   Using \eqref{eq:nk3} to count the total number of such representations without restrictions on the order or inequality of the $y_i$.  

We must exclude from this count any representations having two (or more) equal $y_i$ values.  We obtain an upper bound for the total number of such representations in a manner similar to the observation used in the base case.  We can produce a representation of $n=x_1y_1+ \cdots + x_ky_k$ which has at least two equal $y_i$ values using the following construction.  Fix $2\leq m <n$, and take a representation of $n-m=x_1y_1+ \cdots + x_{k-2}y_{k-2}$ as a sum of $k-2$ products.  This can be done in $\nu_{k-2}(n-m)$ ways.  

Now, add to this representation two additional terms $x_{k-1}y_{k-1}$ and $x_{k}y_{k}$ where $y_{k-1}=y_k$, of total size $m=y_k(x_{k-1}+x_k)$. As in the base case there are $\sigma_1(m)-\sigma_0(m)$ ways to choose the values of $y_k,x_k$ and $x_{k-1}$.  This produces a representation of $n$ as a sum of $k$ products in which the heights in the last two terms are equal.

Finally, we note that every representation of $n$ having at least two equal $y_i$ values can be constructed in this manner by varying $m$ and then permuting the positions of the two equal $y_i$ terms produced by this method among the $k$ indices.  There are $\binom{k}{2}$ choices for where these terms could be inserted into the representation of $n$. 

This method overcounts those representations with more than two equal terms. As we seek only an upper bound, the following estimates which bound the total number of such representations having at least two equal $y_i$ terms are sufficient for our purposes:
\begin{align}
    \binom{k}{2}\sum_{m=1}^{n}\nu_{k-2}(n-m)(\sigma_1(m)-\sigma_0(m)) &\ll_k \sum_{m=1}^{n}\nu_{k-2}(n-m)\sigma_1(m) \nonumber \\
    &\ll_k 
    \begin{cases}  
    \sum\limits_{m=1}^{n} \sigma_{0}(n-m)\sigma_1(m) & k=3
    \\
    \sum\limits_{m=1}^{n} \sigma_{k-3}(n-m)\log^{k-2}(n-m)\sigma_1(m) & k\geq 4.
    \end{cases}
\label{eq:convolution}
\end{align} 

If $k=3$ the sum in \eqref{eq:convolution} is at most $n^2 \log n \log \log n$, obtained using the bound $\sigma_1(m)\ll m \log \log m$ along with the fact that $\sum_{j=1}^n \sigma_0(n)\ll n \log n$.  If $k\geq 4$, the resulting sum above is bounded above by  
\begin{align*}
 \log^{k-2} n \sum_{m=1}^{n}\sigma_{k-3}(n-m)\sigma_1(m)
    \ll_k n^{k-1} \log^{k-2} n.
\end{align*} 
The final bound above is obtained using \cite{halberstam} where it is shown that for any $a,b>0$\[\sum_{m=1}^n \sigma_a(n-m)\sigma_b(m) \ll_{a,b} \sigma_{a+b+1}(n) \ll_{a,b} n^{a+b+1}.\] 

In either case the number of representations having at least two equal $y_i$ terms is at most $n^{k-1}\log^{k-1}$. Thus the number of representations of $n$ as a sum of $k$ distinct products of terms with distinct heights $y_i$ is $\nu_k(n)\left(1+O\left(\frac{1}{\log n}\right)\right)$.

Finally, each representation in which the $y_i$ are strictly decreasing occurs exactly $k!$ different times in this count, once for each potential permutation of the indices, so we can conclude that

\[|D_n(\mu)| = \frac{1}{k!} \nu_k(n)\left(1+O\left(\frac{1}{\log n}\right)\right) = \frac{1}{k!(k-1)!\zeta(k)}\sigma_{k-1}(n)\log^k n\left(1+O\left(\frac{1}{\log n}\right)\right)\] for $k\geq 3$, and the result follows.
\end{proof}

Before closing this section we state and prove a useful corollary for the sequel.  

\begin{corollary}\label{cor:fewerparts}
    Suppose that $\mu$ is strict with $\mu_1>1$. Set $k = \mu_1-1$.   Then $$|\Av_n(\mu)\setminus D_n(\mu)| \ll \begin{cases} \sigma_0(n) & k=2 \\
    \sigma_{k-2}(n)\log^{k-1} n & k \geq 3 \end{cases}.$$ 
\end{corollary}
\begin{proof}
This follows immediately from Theorem~\ref{thm:staircase} since any partition having fewer than $k$ distinct magnitudes avoids the pattern $(k,k-1,\ldots, 1)$.
\end{proof}

\subsection{Avoiding strict partitions $\mu$ that are not staircases}\label{ssec:not staircase}

We now consider avoiding strict partitions that are not staircases.  To do so, we introduce a method of constructing partitions avoiding a pattern $\mu$ from those avoiding a smaller pattern $\hat \mu$.

\begin{definition}
	Let $\mu=(\mu_1,\mu_2,\ldots)$ be a strict partition, and let $i>1$ be the least index such that  $\mu$ does not have a part of size $\mu_i+1$.  If such an index exists, we define 
	$$\hat\mu = (\mu_1-1,\mu_2-1, \ldots \mu_{i-1}-1, \mu_{i+1}, \mu_{i+2},\ldots),$$ 
	i.e., the result of deleting the $i$-th part of $\mu$, as well as decreasing each part above by one.  If no such index exists then $\hat \mu$ is not defined.  
\end{definition}

Graphically, $\hat \mu$ is obtained from $\mu$ by removing an ``L'' shaped region from the Ferrers board for $\mu$.  For example, if $\mu=(7,6,4,2)$, then $\hat\mu = (6,5,2)$ is obtained by removing the boxes in red below:

$$\ydiagram{7,6,4,2}*[*(red)]{0,0,4,0}*[*(red)]{3+1,3+1}\ .$$

Note that the only situation in which $\hat \mu$ is not defined is when $\mu$ is a staircase partition, which was treated in the previous section.   In the case when $\mu=(k+1,k,k-1,\ldots, \ell)$ contains a part of every size from $\ell>1$ to $k+1$, but no smaller parts, then $\hat \mu$ is obtained by removing the first column from $\mu$. 

The remainder of this section is divided into two subsections.  The first subsection deals with the special case when $\mu_1-\mu_2\geq 2$ and the next subsection deals with the general case.  The advantage of first considering this special case is that the constructions involved are simpler and motivate the constructions involved in the general argument.  Furthermore, the error term we obtain in the special case is stronger than its counterpart in the general case.  

\subsubsection{Avoiding partitions $\mu$ where $\mu_1-\mu_2\geq 2$}\label{sssec:nonstaircase 1}
We begin with a motivating example. Take $\mu = (6,3,2,1)$.  By Lemma~\ref{lem:rd of mu} the elements of $D(\mu)$ have the rectangular decomposition
\begin{equation}\label{eq:decomp6321}
	x_1y_1 +x_2y_2 + x_3y_3 + y_4 + y_5.
\end{equation}
Likewise, the partitions in $D(\hat\mu) = D((5,2,1))$  have the rectangular decomposition
 \begin{equation}\label{eq:decomp5321}
x_1y_1 +x_2y_2 + y_3 + y_4.
\end{equation}
We now give a construction to create partitions in  $D(\mu)$ from those in $D(\hat\mu)$. To this end take the rectangular decomposition for partitions in $D(\hat\mu)$ and consider the height of the rectangle corresponding to the part deleted from $\mu$ in the construction of $\hat\mu$.  In this example we deleted 3 from $\mu$ and so we consider $y_3$ in (\ref{eq:decomp5321}).   

Now for any $m>0$ we have, by the division algorithm, 
$$m = y_3d + r$$
for some $d\geq 0$ and  $y_3>r \geq 0$.  Taking a partition in $D(\hat\mu)$ with rectangular decomposition (\ref{eq:decomp5321}) we can now add to it $d$ columns of height $y_3$ and 1 column of height $r$.  Doing this transforms the third rectangle into a wide rectangle and creates a new thin rectangle either before or after the final rectangle with height $y_4$ depending on whether $r>y_4$ (before) or $r\leq y_4$ (after).  Of course, in the case when $r=0$ no new column is added. We thus obtain a partition of weight $n+m$ whose Ferrers board looks like one of the following, where the $m$ additional boxes are shaded gray:

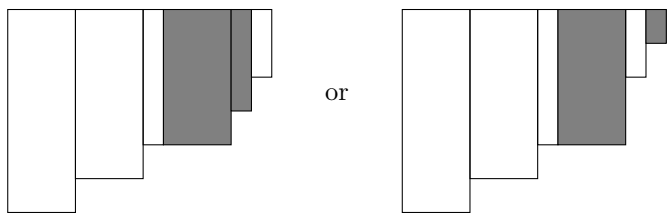
\begin{center}
    \begin{tikzpicture}[scale=0.90]
    \draw (0,0) rectangle (1,3);
    \draw (1,0.5) rectangle (2,3);
    \draw (2,1) rectangle (2.3,3);
    \draw[fill = gray] (2.3,1) rectangle (3.3,3);
    \draw[fill = gray] (3.3,1.5) rectangle (3.6,3);
    \draw (3.6,2) rectangle (3.9,3);
    \end{tikzpicture}\qquad\raisebox{1.5cm}{or}\qquad
\begin{tikzpicture}[scale=0.90]
    \draw (0,0) rectangle (1,3);
    \draw (1,0.5) rectangle (2,3);
    \draw (2,1) rectangle (2.3,3);
    \draw[fill = gray] (2.3,1) rectangle (3.3,3);
    \draw (3.3,2) rectangle (3.6,3);
    \draw[fill = gray] (3.6,2.5) rectangle (3.9,3);
    \end{tikzpicture}
\end{center}
or, in the ``degenerate'' case when either $r=0$ or $r=y_4$ we can also have:
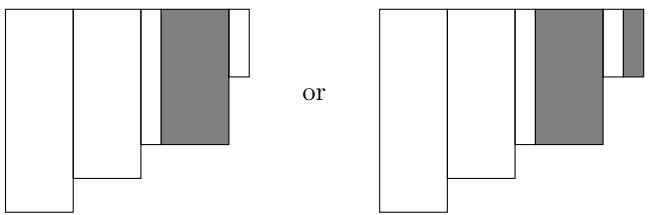
\begin{center}
    \begin{tikzpicture}[scale=0.90]
    \draw (0,0) rectangle (1,3);
    \draw (1,0.5) rectangle (2,3);
    \draw (2,1) rectangle (2.3,3);
    \draw[fill = gray] (2.3,1) rectangle (3.3,3);
    \draw (3.3,2) rectangle (3.6,3);
    \end{tikzpicture}\qquad\raisebox{1.5cm}{or}\qquad
\begin{tikzpicture}[scale=0.90]
    \draw (0,0) rectangle (1,3);
    \draw (1,0.5) rectangle (2,3);
    \draw (2,1) rectangle (2.3,3);
    \draw[fill = gray] (2.3,1) rectangle (3.3,3);
    \draw (3.3,2) rectangle (3.6,3);
    \draw[fill = gray] (3.6,2) rectangle (3.9,3);
    \end{tikzpicture}  \ .  
\end{center}
Note that the first two diagrams depict rectangular decompositions of the form in (\ref{eq:decomp6321}), while the degenerate cases result in partitions with only 4 distinct magnitudes.  We show that the contribution from these degenerate cases are asymptotically negligible.  

In light of this construction we make the following definition. Recall that $\P$ is the set of all partitions.  
 
\begin{definition}
Fix a strict partition $\mu$ for which $\hat\mu$ is defined.  So there exists a least index $i$ greater than 1 such that $\mu$ does not have a part of size $\mu_i+1$ and $\mu_i$ is the part removed from $\mu$ to obtain $\hat\mu$.  Now for each $m>0$ we define a function $\Psi_m:D(\hat\mu)\to \P$ as follows.

Fix $\alpha\in D(\hat\mu)$.  If $\mu_i>0$ then let $q$ be the height of the $\mu_i$-th rectangle in the rectangular decomposition for $\alpha\in D(\hat\mu)$. (By Lemma~\ref{lem:rd of mu} we know that $q>0$ and the corresponding rectangle is thin.) By the division algorithm write 
$$m = qd +r$$ 
for some $d\geq 0$ and $q > r \geq 0$. If $\mu_i=0$ then we take $r = m$ and $d=0$. Finally define $\Psi_m(\alpha)$ to be the partition obtained by adding $d$ columns of height $q$ along with a single column of height $r$ to $\alpha$.

Note that the location of the inserted column of height $r$ depends on the heights of the existing rectangles in $\alpha$ and may have the same height as one of these existing rectangles.
\end{definition}

\begin{remark}
We caution the reader that the function $\Psi_m$ depends on $\mu$.  As a result proper notation should reflect this fact, e.g., one could instead denote this function as $\Psi_{m,\mu}$.  To streamline notation though, we have chosen to denote this function as above since the $\mu$ in question will always be clear from context. 
\end{remark}

\begin{lemma}

\label{lem:muhat1}
Take $k\geq 3$ and suppose $\mu=(k+1, a_0, a_1,\ldots)$ is a strict partition where $k>a_0>0$.  Then
$$\sum_{m=1}^{n} |D_{m}(\hat\mu)| = (k-a_0)|D_n(\mu)| + 
    O\left(\sigma_{k-2}(n)\log^{k-1} n  \right).$$
\end{lemma}
\begin{proof}
In this case $\hat\mu = (k, a_1,\ldots).$  For $1 \leq m <n$ we know, by Lemma~\ref{lem:rd of mu},  that any $\alpha \in D_{n-m}(\hat \mu)$ has the rectangular decomposition
$$ n-m = \cdots + x_{a_1}y_{a_1} + y_{a_1+1} + \cdots +y_{a_0}+ \cdots + y_{k-1},$$
with $k>a_0>0$ and so $y_{a_0}>0$. We now show that for each such $\alpha$ we have $\Psi_{m}(\alpha) \in D_n(\mu) \cup \Av_n\big((k,k-1,\ldots, 1)\big)$. To this end write $m= y_{a_0}d+r$ where $d\geq 0$ and $y_{a_0}>r\geq 0$.  We consider two cases depending on the value of $r$. 

\medskip\noindent
\textbf{Case 1: $r \in \{y_{a_0+1}, \ldots, y_{k-1}\}\cup \{0\}$}
\medskip

In this case $\Psi_m(\alpha)$ has exactly $k-1$ distinct magnitudes and hence $\Psi_m(\alpha)\in \Av_n(k,k-1,\ldots, 1)$.  Moreover, for each $\beta \in \Av_n(k,k-1,\ldots, 1)$ there can be at most $k$ partitions $\alpha \in \bigcup_{m=1}^{n}D_{n-m}(\mu)$ with $\Psi_{n-|\alpha|}(\alpha) =\beta$.  

\medskip\noindent
\textbf{Case 2: $r \notin \{y_{a_0+1}, \ldots, y_{k-1}\}\cup \{0\}$}
\medskip

Let $i\geq a_0$ be such that $y_i> r >y_{i+1}$ where we set $y_k:= 0$. In this case $\Psi_m(\alpha)$ has the rectangular decomposition
$$ n =  \cdots +(d+1)y_{a_0}+ y_{a_0+1}+ \cdots +y_i +r+ y_{i+1}+\cdots + y_{k-1},$$
where we have only displayed the pertinent terms.  By Lemma~\ref{lem:rd of mu} we see that $\Psi_m(\alpha) \in D(\mu)$.

We also show that for each $\beta\in D_n(\mu)$, there are $(k-a_0)$ partitions $\alpha \in \bigcup_{m=1}^{n}D_{n-m}(\hat \mu)$ with $\Psi_{n-|\alpha|}(\alpha) = \beta$.  In particular such a  $\beta$ has the rectangular decomposition

$$ n =  \cdots +x_{a_0}y_{a_0}+y_{a_0+1}+ \cdots  + y_{k}. $$

Set $m = (x_{a_0}-1)y_{a_0} + r$ where $r \in\{y_{a_0+1},\ldots, y_{k}\}$ and define $\alpha$ to be the partition obtained by deleting $x_{a_0}-1$ columns of height $y_{a_0}$ and the column of height $r$ from $\beta$.  It now follows from our definitions that $\Psi_{m}(\alpha) = \beta$ and thus there exists exactly $(k -a_0)$ partitions $\alpha\in D(\hat\mu)$ with $\Psi_{n-|\alpha|}(\alpha) = \beta$, one for each choice of $r$.

\medskip

Now, applying the map $\Psi_m$ to the partitions in $D_{n-m}(\hat \mu)$ and counting the partitions so created we have
$$\sum_{m=1}^{n} |D_m(\hat\mu)| = (k-a_0)|D_n(\mu)|+ O\left(k\cdot \left|\Av_n\big((k,k-1,\ldots, 1)\big) \right|\right).$$
As $k\geq 3$, we see by Theorem \ref{thm:staircase} that $k\cdot \left|\Av_n \big((k,k-1,\ldots, 1)\big) \right| \ll \sigma_{k-2}(n)\log^{k-1} n$
 which completes our proof. 
\end{proof}

\noindent We are now able to count  partitions avoiding a strict partition $\mu$ whose first two parts differ by at least 2.

\begin{theorem} \label{thm:nonotch}
Consider a strict partition $\mu=(k+1,a_0,a_1,\ldots)$ with $k>a_0$.  Then \begin{equation}
    |\Av_n(\mu)| = \frac{n^{k-1}}{(k-1)!\prod_{j=0}^{k-1}\left(k-(a_j+j)\right)}+E(n). \label{eq:nonotch}
\end{equation}
where the error term, $E(n)$ satisfies
\begin{equation}
    E(n) = \begin{cases}
    0 & k=1 \\
    O(1) & k=2 \\
    O(\sigma_{k-2}(n)\log^{k-1} n) & k\geq 3.
    \end{cases} \label{eq:errorterms}
\end{equation}

\end{theorem}
\begin{proof}
First note that when $k=1$ or 2, the only possibilities are $\mu=(2), (3)$ and $(3,1)$.  These cases were all treated in Theorem \ref{thm:sporadic}, and a quick check shows that they agree with the formula above with the error terms as in \eqref{eq:errorterms}.  For the remainder of the proof we can assume $k\geq 3.$

We now prove this by induction on the number of nonzero parts in $\mu$.
The base case, when $\mu$ has one part (meaning $a_0=0$ and $\mu=(k+1)$) occurs when the set $\Av_n(\mu)$ is precisely the set of partitions of $n$ into parts of size at most $k$.   It is well known since at least Sylvester that the count of such partitions is asymptotic to $\frac{n^{k-1}}{k!(k-1)!}$ (see for example \cite{Ramirez}).  Nathanson \cite{Nathanson} reproves this result with a power saving error term, which implies that 
\begin{equation*}
|\Av_n\big((k+1)\big)| = \frac{n^{k-1}}{k!(k-1)!} + O\left(n^{k-2}\right), 
\end{equation*} 
which is equation \eqref{eq:nonotch} in the case when $a_i=0$ for all $i$ (with an even stronger error term).

Now suppose $\mu=(k+1,a_0,\ldots)$ satisfies the hypotheses of the theorem, that $k>a_0>0$ and assume the result holds when avoiding any such partition with fewer parts than $\mu$.  By Corollary~\ref{cor:fewerparts} the number of partitions having fewer than $k$ distinct magnitudes is $O\big(\sigma_{k-2}(n)\log^{k-1}n \big)$, the size of our error term.  Thus we can restrict ourselves to counting $D_n(\mu)$.

Note that $\hat \mu$ satisfies the hypotheses of the theorem, and has one fewer part than $\mu$, so by induction
\begin{align*}
    |\Av_{m}(\hat\mu)| &= \frac{m^{k-2}}{(k-2)!\prod_{j=1}^{k-1}\left(k-a_j-j\right)} + \begin{cases}
      O(1) & k=3 \\
      O\left(\sigma_{k-3}(m)\log^{k-2}m\right)&k\geq 4
    \end{cases}
\end{align*}
where we start our indexing at $j=1$ in the product to account for the deletion of the part $a_0$ from $\mu$.  Further we know that $|D_m(\hat\mu)| = |\Av_m(\hat\mu)| + O\left(D_m(\hat\mu) \setminus \Av_m(\hat\mu)\right)$
where by  Corollary \ref{cor:fewerparts} the error term is $O(\sigma_0(m))$ when $k=3$ and $O(\sigma_{k-3}(m)\log^{k-2}m)$ when $k\geq 4$.  Putting this together we have
\begin{align}
    |D_{m}(\hat\mu)| &= \frac{m^{k-2}}{(k-2)!\prod_{j=1}^{k-1}\left(k-a_j-j\right)} + \begin{cases}
O\left(\sigma_{0}(m)\right)&k=3 \\
O\left(\sigma_{k-3}(m)\log^{k-2}m\right)&k\geq 4. \label{eqn:muhatcount}
    \end{cases}
\end{align}  
As $k\geq 3$ and $a_0>0$ we also know from Lemma~\ref{lem:muhat1} that
\begin{equation}\label{eqn:mapsum}
    \sum_{m=1}^{n} |D_{m}(\hat\mu)| = (k-a_0)|D_n(\mu)| + 
    O\left(\sigma_{k-2}(n)\log^{k-1} n  \right).
\end{equation}

Inserting \eqref{eqn:muhatcount} into \eqref{eqn:mapsum} and solving for $|D_n(\mu)|$ we have
\begin{align*}
    |D_n(\mu)| &= \frac{1}{k{-}a_0} \sum_{m=1}^n |D_m(\hat\mu)|  + 
    O(\sigma_{k-2}(n)\log^{k-1} n) \\
    &= \frac{1}{k{-}a_0} \sum_{m=1}^n \frac{m^{k-2}}{(k{-}2)!\prod_{j=1}^{k-1}\left(k{-}a_j{-}j\right)}+\begin{cases}
    O\left(\sum\limits_{m=1}^n \sigma_0(m)\right) +O(\sigma_1(n)\log^2 n) & k=3 \\
    O\left(\sum\limits_{m=1}^n \sigma_{k-3}(m)\log^{k-2} m\right){+}O(\sigma_{k-2}(n)\log^{k-1} n) & k\geq 4
    \end{cases} \\
    &=\frac{1}{k{-}a_0}\left(\frac{n^{k-1}}{(k{-}1)!\prod_{j=1}^{k-1}\left(k-a_j-j\right)} +O(n^{k-2})\right) + \begin{cases}
    O(\sigma_1(n)\log^2 n) & k=3 \\
    O\left(n^{k-2}\log^{k-1} n\right) & k\geq 4
    \end{cases} \\
    &=\frac{n^{k-1}}{(k-1)!\prod_{j=0}^{k-1}\left(k-a_j-j\right)}+E(n).
\end{align*}
The error terms above were obtained using the facts in \eqref{eq:sigma0}, \eqref{eq:sigma1}, and \eqref{eq:sigmak}. \end{proof}

\subsubsection{The general case}\label{sssec:general}

Here we adapt the ideas of the previous section to handle the case when $\mu$ is any strict partition that is not the staircase. As discussed at the start of Subsection~\ref{ssec:not staircase} the results obtained in the previous subsection are a special case of the result in this subsection; however, the error term obtained here is weaker than what was obtained in the previous section.  We follow a similar strategy to that of the previous subsection. Starting with the same function $\Psi_m(\mu) : D(\hat \mu) \to \mathcal{P}$ we prove a more general form of Lemma \ref{lem:muhat1}.

\begin{lemma}\label{lem:decomp a>0} Take $k\geq 3$ and suppose $\mu$ is strict but not the staircase so that we can write  
$$\mu = (k+1,k,\ldots, k-\ell+1, a_0, a_1,\ldots),$$
where $k-\ell> a_0 \geq 0$, so $k-\ell$ is the largest part size less than $\mu_1=k+1$ omitted from $\mu$, and $a_0,a_1 \ldots$ are the parts of $\mu$ of size less than $k-\ell$. For each $i$ with $ k - \ell < i \leq k$ denote by 
$$\mu^{(i)} = (k+1, \ldots, \widehat i,\ldots, k-\ell, a_0, a_1,\ldots),$$ 
the partition obtained from $\mu$ by removing the part of size $i$ and adding a part of size $k-\ell$, and let ${E_n(\mu) = \bigcup_{i = k - \ell+1}^k D_n(\mu^{(i)})}$. Then $$\sum_{m=1}^{n} |D_{m}(\hat\mu)| = (k-\ell-a_0)|D_n(\mu)| + 
    O\left(\sigma_{k-2}(n)\log^{k-1} n +\left|E_n(\mu)\right| \right).$$
\end{lemma}
\begin{proof}
In this case $\hat\mu = (k,k-1,\ldots, k-\ell, a_1,\ldots).$  For any $1 \leq m \leq n$ we know, by Lemma~\ref{lem:rd of mu},  that any $\alpha \in D_{n-m}(\hat \mu)$ as the	rectangular decomposition $$ n-m = \cdots + x_{a_1}y_{a_1} + y_{a_1+1} + \cdots +y_{a_0}+ \cdots + y_{k-\ell-1}  + x_{k-\ell}y_{k-\ell}+\cdots + x_{k-1}y_{k-1}.$$
We now show that for each such $\alpha$ we have $\Psi_{m}(\alpha) \in D_n(\mu) \cup E_n(\mu) \cup \Av_n(k,k-1,\ldots, 1)$. 

To this end recall the definition of $\Psi_m$ and set 
$$m = dy_{a_0}+r$$
where $0\leq r < y_{a_0}$.  (Recall, if $a_0=0$ then we take $d=0$ and $r=m$.) So $\Psi_m(\alpha)$ is obtained by adding $d$ columns of height $y_{a_0}$ and a single column of height $r$.  We now consider several cases depending on the value of $r$.   For completeness in Cases 2 and 3 we set $y_k:= 0$ and $y_0:=\infty$.

\medskip\noindent
\textbf{Case 1: } $r \in \{y_1, \ldots, y_{k-1}\}\cup \{0\}$
\medskip

In this case $\Psi_m(\alpha)$ has exactly $k-1$ distinct magnitudes and hence $\Psi_m(\alpha)\in \Av_n(k,k-1,\ldots 1)$.  Furthermore, for any $\beta \in \Av_n(k,k-1,\ldots 1)$ there can be at most $k$ partitions $\alpha \in D(\mu)$ with $\Psi_{n-|\alpha|}(\alpha) =\beta$.  This follows since, in particular, there is at most one per distinct height present in $\beta$.

\medskip\noindent
\textbf{Case 2:} $y_i> r >y_{i+1}$ with $k-\ell \leq i$
\medskip

In this case  $\Psi_m(\alpha)$ has rectangular decomposition
$$\cdots +(d+1)y_{a_0}+ \cdots +x_{k-\ell}y_{k-\ell} + \cdots + x_{i}y_{i}+r+x_{i+1}y_{i+1}+\cdots +x_{k-1}y_{k-1}$$
where we have only displayed the pertinent changes to the rectangular decomposition for $\alpha$.  (Note that the decomposition has $k$ rectangles, since the thin rectangle of height $r$ is a new height not present before in the $(i+1)$-st position.)  It then follows by Lemma~\ref{lem:rd of mu} that $\Psi_m(\alpha)\in D(\mu^{(i+1)})\subset E(\mu)$.  Furthermore we see, in this case, that for any $\beta \in D(\mu^{(i+1)})$ with rectangular decomposition 
$$ \cdots + x_{a_0}y_{a_0} +  \cdots + x_{k-\ell}y_{k-\ell} + \cdots +x_{i}y_{i} + y_{i+1} + x_{i+2}y_{i+2} + \cdots + x_ky_k$$  
there is a unique choice of $m=(x_{a_0}-1)y_{a_0}+ y_{i+1}$ and partition $\alpha\in D(\mu)$, obtained by removing $(x_{a_0}-1)$ columns of height $y_{a_0}$ and the single column of height $y_{i+1}$, such that $\Psi_m(\alpha) = \beta$. 

\medskip\noindent
\textbf{Case 3:} $y_i> r >y_{i+1}$ with $i< k-\ell$
\medskip

In this case $\Psi_m(\alpha)$ has the rectangular decomposition
$$ \cdots +(d+1)y_{a_0}+ \cdots +y_i +r+ y_{i+1}+\cdots + y_{k-\ell-1}+ x_{k-\ell}y_{k-\ell} + \cdots + x_{k-1}y_{k-1}$$
where again we only display the pertinent terms (and the term for $y_{a_0}$ is to be ommited when $a_0=0$).  By Lemma~\ref{lem:rd of mu} we see that $\Psi_m(\alpha) \in D(\mu)$.

Lastly, for each $\beta\in D(\mu)$, there are $(k-\ell -a_0)$ partitions $\alpha \in D(\hat \mu)$ with $\Psi_{|\beta|-|\alpha|}(\alpha) = \beta$.  To see this observe that such a  $\beta$ has the rectangular decomposition 
$$\cdots +x_{a_0}y_{a_0}+y_{a_0+1}+ \cdots  + y_{k-\ell-1} + x_{k-\ell+1}y_{k-\ell+1} + \cdots+ x_{k-1}y_{k-1}$$
and we can choose $\alpha$ to be the partition obtained by deleting $x_{a_0}-1$ columns of height $y_{a_0}$ and a single column of height  $r \in\{y_{a_0+1},\ldots, y_{k-\ell}\}$.  

\medskip

Now, applying the map $\Psi_{n-m}$ to the partitions in $D_m(\hat \mu)$ and counting the partitions so created we have
$$\sum_{m=1}^{n} |D_m(\hat\mu)| = (k-\ell-a_0)|D_n(\mu)| + |E_n(\mu)| + O\left(k\cdot \left|\Av_n\big((k,k-1,\ldots 1)\big) \right|\right).$$
By Theorem \ref{thm:staircase} we know that $k\cdot \left|\Av_n \big((k,k-1,\ldots 1)\big) \right| \ll \sigma_{k-2}(n)\log^{k-1} n$
 which completes our proof. 
\end{proof}



\begin{theorem}\label{thm:not staircase}
Suppose $\mu$ is a strict partition that is not a staircase so that 
$$\mu = (k+1,k, k-1,\ldots, k-\ell+1, a_0,a_1,\ldots) $$
where $k-\ell > a_0\geq 0$.  Then 
$$|\Av_n(\mu)| 
= \frac{n^{k-1}\log^\ell n}{\ell!(k-1)!
\prod_{j=0}^{k-\ell-1}\left(k-\ell -a_j-j\right)}
\left(1+O\left(\frac{1}{\log n}\right)\right).$$
\end{theorem}

\begin{proof}
When $k=1$ the only possibility for $\mu$ is (2).  As in Theorem \ref{thm:sporadic} we find that $\left|\Av\big((2)\big)\right|=1$, which trivially matches the expression above.  When $k=2$, the possibilities for $\mu$ are $(3)$, (3,1) and (3,2).  Again, these were counted in Theorem \ref{thm:sporadic}, and the results again fit the statement of the theorem.  Thus for the remainder of the proof we assume $k\geq 3$.

We now proceed by induction on $|\mu|$.  From above we know the result holds for all partitions with $k< 3$.  Now take $\mu$ as in the statement of the theorem and assume the result holds for all partitions with weight $<|\mu|$.  We may further assume that $k\geq 3$. We know that $$\hat\mu =(k,k-1,k-2,\ldots, k-\ell, a_1, a_{2},\ldots).$$
As $|D_m(\hat\mu)| = |\Av_m(\hat\mu)| + O\left(|\Av_m(\hat\mu)\setminus D_m(\hat\mu)|\right)$ it follows by Corollary \ref{cor:fewerparts} that 
\begin{equation} |D_m(\hat\mu)|  = 
|\Av_m(\hat\mu)| + 
\begin{cases}
O(\sigma_0(m)) &\text{ if } k = 3\\
O(\sigma_{k-3}(m)\log^{k-2}m) &\text{ if } k \geq 4.
\end{cases} \label{eq:muhaterror}
\end{equation}

We now wish to find the size of $\sum_{m=1}^{n} |D_m(\hat\mu)|$, but to do so we must consider three cases.

\medskip\noindent
\textbf{Case 1:} $k-l\geq 2$
\medskip

In this case we know that $\hat\mu$ is not a staircase since $k-\ell \geq 2$ and $a_1 \leq a_0 < k-\ell -1$ and so $\hat \mu$ is missing a part of size $k-\ell -1>0$.  Furthermore the weight of $\hat\mu$ is clearly smaller than $\mu$ so we may apply our induction hypothesis, along with \eqref{eq:muhaterror} to obtain
\begin{align}
|D_m(\hat\mu)| &=\frac{m^{k-2}\log^\ell m}{\ell!(k-2)!
\prod_{j=0}^{k-1-\ell -1} ((k-1)-\ell - a_{j+1} - j)}+O\left(m^{k-2}\log^{\ell-1}m\right) \nonumber \\
& =\frac{m^{k-2}\log^\ell m}{\ell!(k-2)!
\prod_{j=1}^{k-\ell-1}\left(k-\ell-a_{j}-j\right)}+O\left(m^{k-2}\log^{\ell-1}m\right).
\end{align}
Summing this expression over all $m$ from 1 to $n$ and making use of \eqref{eq:logsum} we find
\begin{align}
	\sum_{m=1}^{n} |D_m(\hat\mu)| 
	&= \sum_{m=1}^{n}\left( \frac{m^{k-2}\log^\ell m}{\ell!(k-2)!
\prod_{j=1}^{k-\ell-1}(k-\ell-a_j-j)}\right)
+ O\left(\sum\limits_{m=1}^n m^{k-2}\log^{\ell-1} m\right)  \nonumber \\
	&=\frac{n^{k-1}\log^\ell n}{\ell!(k-1)!
\prod_{j=1}^{k-\ell-1}\left(k-\ell-a_{j}-j\right)} + O\left(n^{k-1}\log^{\ell-1} n\right). \label{eq:case1sum}
\end{align}

\medskip\noindent
\textbf{Case 2:} $k-\ell = 1$, $k\geq 4$
\medskip

Here $\mu = (k+1,\ldots, 2)$  and so $\hat \mu = (k,k-1,\ldots, 1)$ is a staircase. This case proceeds similarly to the former case, but instead of induction we must use Theorem \ref{thm:staircase} to find the size of $D_m(\hat \mu)$.  For $k\geq 4$ we have
\begin{align*}
    |D_{m}(\hat\mu)| &= 
\frac{\sigma_{k-2}(m)\log^{k-1} m}{(k-1)!(k-2)! \zeta(k-1)}\left(1 + O\left(\frac{1}{\log m}\right)\right).
\end{align*}
Summing now this expression using \eqref{eq:sigmalogsum} we have
\begin{align}
	\sum_{m=1}^{n} |D_m(\hat\mu)| 
	&= \sum_{m=1}^{n} \frac{\sigma_{k-2}(m)\log^{k-1} m}{(k-1)!(k-2)! \zeta(k-1)}\left(1 + O\left(\frac{1}{\log m}\right)\right) \nonumber \\
	&=\frac{n^{k-1}\log^{k-1} n}{(k-1)!(k-1)!} + O\left(n^{k-1}\log^{k-2} n\right) \nonumber \\
	&=\frac{n^{k-1}\log^\ell n}{\ell!(k-1)!
\prod_{j=1}^{k-\ell-1}\left(k-\ell-a_{j}-j\right)} + O\left(n^{k-1}\log^{\ell-1} n\right) \label{eq:case2sum}
\end{align}
where the product written in the denominator of the final term above is empty, but included so as to be written in the same form as \eqref{eq:case1sum}.

\medskip
\noindent\textbf{Case 3:} $k-\ell = 1$, $k=3$
\medskip

As in the previous case, $\hat \mu$ is a staircase, however since $\hat \mu=(3,2,1)$, we need to work a little harder to get the same error term. In particular we use the more precise expression given in \eqref{eq:nk2-precise}
\begin{align}
    \left|D_m\big((3,2,1)\big)\right| =\frac{1}{2\zeta(2)}\sigma_{1}(m)\log^2 m-\frac{2}{\zeta(2)} m \log m \sigma_{\hspace{-0.7mm}\shortminus 1} '(m) + O\left(\sigma_1(m) \log m\right). \label{eq:a321again}
\end{align}
Summing the main and error terms of \eqref{eq:a321again} over $m$ from $1$ to $n{-}1$ gives the same result as in \eqref{eq:case2sum}, by the same argument, so we treat only the sum of the second term, $\frac{2}{\zeta(2)} m \log m \sigma_{\hspace{-0.7mm}\shortminus 1} '(m) = \frac{2}{\zeta(2)} m \log m \sum\limits_{d|m} \frac{\log d}{d}$. 

\begin{align*}
    \sum_{m=1}^{n} \left(\frac{2}{\zeta(2)} m \log m \sum_{d|m} \frac{\log d}{d}\right) &= \frac{2}{\zeta(2)}\sum_{d<n}\left( \frac{\log d}{d} \sum_{c=1}^{\left\lfloor\frac{n}{d} \right \rfloor} cd \log(cd)\right)\\
    &\ll \sum_{d<n}\left(\log d \cdot \left(\frac{n}{d}\right)^2\cdot  \log \left(\frac{n}{d}\cdot d\right)   \right)\\
    &\ll n^2 \log n \int_1^n \frac{\log t}{t^2}\ dt\\
    &\ll n^2 \log n.
\end{align*}
Thus we find that the contribution from this term can be absorbed into the error term in \eqref{eq:case2sum}, and we obtain the same result for $k=3$ as well.
\medskip

Since we obtained the same result \eqref{eq:case1sum} and \eqref{eq:case2sum} in all three cases, we now proceed using that expression for the sum of $|D_m(\hat \mu)|$. Note that any partition of the form
$$\mu^{(i)} = (k+1, \ldots, \widehat i,\ldots, k-\ell, a_0, a_1,\ldots),$$ 
where $k-\ell < i \leq k$ and $\widehat{\phantom{x}}$ denotes deletion, has weight smaller than $\mu$.  By induction we have $$\left|\Av_n\left(\mu^{(i)}\right)\right| \ll n^{k-1}\log^{k-i}n.$$  Therefore we may apply Lemma~\ref{lem:decomp a>0} to obtain
\begin{align*} 
\sum_{m=1}^n |D_{m}(\hat\mu)| &= (k-\ell-a_0)|D_n(\mu)| + \sum_{i=k-\ell+1}^k |D_n(\mu^{(i)})| + O\left(\sigma_{k-2}(n)\log^{k-1} n \right) \\
&= (k-\ell-a_0)|D_n(\mu)| + 
    O\left(n^{k-1}\log^{\ell-1} n\right). 
\end{align*}
Solving for $|D_n(\mu)|$ above and combining that with the sum obtained in either \eqref{eq:case1sum} or \eqref{eq:case2sum} we obtain
\begin{align*}
    |D_n(\mu)| &= 
\frac{1}{k-\ell-a_0} \sum_{m=1}^n |D_m(\hat\mu)|  + 
    O\left(n^{k-1}\log^{\ell-1} n\right) \\
&=\frac{n^{k-1}\log^{\ell}n}{\ell!(k-1)!
    \prod_{j=0}^{k-\ell-1}\left(k-\ell -a_j-j\right)}
+ O\left(n^{k-1}\log^{\ell-1} n\right)\\
&=\frac{n^{k-1}\log^{\ell}n}{\ell!(k-1)!
    \prod_{j=0}^{k-\ell-1}\left(k-\ell -a_j-j\right)}
    \left(1+ O\left(\frac{1}{\log n}\right)\right).
\end{align*}
This completes the proof.  
\end{proof}

\begin{corollary}\label{cor:not rational}
    If $\mu$ is a strict partition with $\mu_1-\mu_2 =1$ and $\mu_2>0$ then the generating function for $\mu$ is not algebraic.  
\end{corollary}
\begin{proof} 
When $\mu$ is either not a staircase (with $\mu_1-\mu_2=1$) or is a staircase $\mu=(k+1,k,\ldots, 1)$ with $k \geq 3$ we know from Theorems~\ref{thm:not staircase} and \ref{thm:staircase} respectively that \begin{equation}|\Av_n(\mu)|\asymp n^{k-1}\log^\ell n \label{eq:approxassymp} \end{equation} for some $\ell>0$. The essence of the proof is that it is not possible for the coefficients of rational generating functions to have such logarithmic factors combined with \begin{quote}
\vspace{-\baselineskip}\begin{theorem}[Fatou's Theorem \cite{fatou}]  If $G(z) = \sum_{i=1}^\infty a_n z^n \in \mathbb{Z}[[z]] $ converges inside the
unit disk, then either $G(z) \in \mathbb{Q}(z)$ or $G(z)$ is transcendental over $Q(z)$. Moreover, if $G(z)$ is rational, then each pole is located at a root of unity.

\end{theorem}
\end{quote}

Suppose that the generating function $F_\mu(x)$ of such a sequence were rational, $F_\mu(x) = \frac{p(x)}{q(x)}$.  By Theorem 4.1.1 of \cite{EC1} we have \begin{equation}|\Av_n(\mu)| = \sum_i P_i(n) \lambda_i^n\ \label{eq:stanley} \end{equation} where $\left(\frac{1}{\lambda_i}\right)$ are the roots of $q(x)$, and the $P_i(n)$ are polynomials.  Since $\lim\limits_{n \to \infty} |\Av_n(\mu)|^{1/n} = 1$, we have that the largest value of $|\lambda_i|$ (respectively the norm of the smallest root of $q(x)$) is 1 and $F_\mu(x)$ has radius of convergence 1, so Fatou's theorem applies.

If there were a unique such root of norm 1, we would be done, as the dominant terms of \eqref{eq:approxassymp} and \eqref{eq:stanley} do not agree.  Otherwise, if $q(x)$ has multiple distinct roots, $\lambda_1,\lambda_2, \ldots, \lambda_j$ of norm 1, Fatou's theorem tells us these roots must all be roots of unity.  Let $m$ be such that $\lambda_i^m=1$ for all $1\leq i \leq j$.  Then for multiples of $m$ we have that the dominant term of \eqref{eq:stanley} is $\sum_{i=1}^j P_i(nm) \asymp_m n^a$ for some integer $a$, whereas the dominant term of \eqref{eq:approxassymp}  is $(nm)^{k-1}\log^\ell (nm) \asymp_m n^{k-1}\log n$.  Thus the generating function cannot be rational, and hence by Fatou's theorem not algebraic.  

When $\mu=(3,2,1)$ Theorem~\ref{thm:staircase} implies that $n\log^2 n \ll |\Av_n(\mu)| \ll n\log^2 n \log \log n$, and this also cannot be rational (or algebraic) by the same argument.  Finally when $\mu=(2,1)$ and $|\Av_n(\mu)|=\sigma_0(n)$, whose generating function is well known to have a natural boundary on the unit circle, and thus is not algebraic.

\end{proof}

\pagebreak

\section*{Table of small partitions}
\begin{table}[!h]
    \centering
    \begin{tabular}{|c|c|c|c|c|}
    \hline
        $k$ & $\mu$ & $F_\mu(z)$ & $|\Av_n(\mu)|$ & OEIS \\
        \hline
        0 & (1) & 0 & 0 & - \\
        1 & (2) & $\frac{1}{1-z}$ & 1 & \href{http://oeis.org/A000012}{A000012}\\
          & (2,1) & - - & $\sigma_0(n)$ & \href{http://oeis.org/A000005}{A000005}\\
        2 & (3) & $\frac{1}{(1-z)(1-z^2)}$& $\left\lfloor\frac{n}{2}\right\rfloor+1$  & \href{http://oeis.org/A004526}{A004526}\\
          & (3,1) & $\frac{1}{(1-z)^2} $ & $n$ & \href{http://oeis.org/A000027}{A000027}\\
          & (3,2) & - -  & $n \log n +(2\gamma{-}2)n +O\left(n^{\frac{131}{416}}\right)$ &  \href{http://oeis.org/A320226}{A320226}\\
          & (3,2,1) & - -  &  $\frac{\sigma_{1}(n)\log^2 n}{2\zeta(2)}  +\frac{2n\sigma_{\hspace{-0.3mm}\shortminus 1}'(n)\log n }{\zeta(2)} + O(\sigma_1(n)\log n)$ & \href{http://oeis.org/A265250}{A265250} \\
        3  & (4) & $\frac{1}{(1-z)(1-z^2)(1-z^3)}$ & $\left[\frac{n^2+6n+9}{12}\right]$ & \href{http://oeis.org/A001399}{A001399}\\
          &  (4,1) & $\frac{z(z^2 - z - 1)}{(z - 1)^3(z + 1)^2}$& $\frac{2n^2 + 10n + 3 + (-1)^n (2n - 3)}{16}$&\href{http://oeis.org/A117142}{A117142} \\
          &  (4,2) & $\frac{1-z+z^3}{(1-z)^2(1-z^2)}$ & $\left\lceil\frac{n^2+3}{4}\right\rceil = \frac{n^2}{4} + \frac{7 + (-1)^n}{8}$ & \href{http://oeis.org/A033638}{A033638}\\
          &  (4,2,1) & - - & $\frac{n^2}{2} {-} n\log n{+}\left(\frac{3}{2}{-}2\gamma\right)n {+}O\left(n^{\frac{131}{416}}\right)$ & \href{http://oeis.org/A309097}{A309097}\\
          &  (4,3) & - -& $\frac{n^2}{4} \log n-\left(\frac{9}{8} -\frac{\gamma}{2}\right)n^2 +O\left(n^{\frac{3}{2}}\right)$ &\href{http://oeis.org/A309098}{A309098}\\
          &  (4,3,1) & - - & $\frac{n^3 \log n}{2} +O(n^3)$& \href{http://oeis.org/A309099}{A309099}\\
          &  (4,3,2) & - -& $\frac{n^3 \log ^2 n}{4} + O(n^3 \log n)$&  \href{http://oeis.org/A309194}{A309194}\\
          &  (4,3,2,1) & - - & $\frac{\sigma_2(n)\log^3 n}{6\zeta(3)} + O\left(n^2\log^2 n\right)$ & \href{http://oeis.org/A309058}{A309058} \\ 
          4 & (5) & $\frac{1}{(1-z)(1-z^2)(1-z^3)(1-z^4)}$& $\left[\frac{n^3 + 15n^2 +\left(\frac{135 +9(-1)^n}{2}\right)n  + 94 + 18(-1)^n}{144}\right]$ & \href{http://oeis.org/A001400/}{A001400}\\
           & (5,1) & $\frac{z(z^5-z^4-z^3+z+1)}{(z-1)^4(z+1)(z^2+z+1)^2}$ & $ \frac{\left(2 \, n^{2} - {\left(10 \, n + 21\right)} \left \lfloor \frac{n + 2}{3} \right \rfloor + 14 \, \left \lfloor \frac{n + 2}{3} \right \rfloor^{2} + 10 \, n + 14\right)}{8} \left \lfloor \frac{n + 2}{3} \right \rfloor +O(1)$ & \href{http://oeis.org/A117143}{A117143}\\
          & (5,2) & $\frac{-z(z^7 - 2z^5 + z^3 + z^2 - z - 1)}{(z - 1)^4(z + 1)^2(z^2 + z + 1)}$ & $\frac{n^3 + 12n^2 + \left(\frac{15+ 9(-1)^n}{2}\right)n}{72} + O(1)$ & \href{http://oeis.org/A136185}{A136185} \\

          \hline
    \end{tabular}
    \caption{Table of avoidance statitistics for small, strict, partitions $\mu$.}
    \label{tab:my_label}
\end{table}

\renewcommand{\biblistfont}{\normalfont\normalsize}
\bibliographystyle{amsplain}
\bibliography{mybib}

\end{document}